\documentclass[12pt, a4paper]{amsproc}

\usepackage{setspace}
\usepackage{latexsym}
\usepackage{amssymb}
\usepackage{amsthm}
\usepackage{amsfonts}
\usepackage{amsmath}
\usepackage[margin = 2.5cm,marginparwidth=2.4cm, marginparsep=0mm]{geometry}
\usepackage{enumerate}
\usepackage{url, hyperref}
\hypersetup{citecolor=blue, linkcolor=blue, colorlinks=true}
\usepackage{longtable, float, booktabs, cancel}

\newtheorem{thm}{Theorem}[section]

\newtheorem{conj}[thm]{Conjecture}
\newtheorem{lem}[thm]{Lemma}

\newtheorem{hyp}[thm]{Hypothesis}
\newtheorem{prop}[thm]{Proposition}

\newcommand{\Aut}{{\rm Aut}}
\newcommand{\Sz}{{\rm Sz}}

\newcommand{\PSL}{{\rm PSL}}
\newcommand{\SL}{{\rm SL}}
\newcommand{\GL}{{\rm GL}}
\newcommand{\SO}{{\rm SO}}
\newcommand{\PGL}{{\rm PGL}}
\newcommand{\AGL}{{\rm AGL}}
\newcommand{\ASL}{{\rm ASL}}
\newcommand{\PSU}{{\rm PSU}}
\newcommand{\SU}{{\rm SU}}

\newcommand{\Out}{{\rm Out}}
\newcommand{\Ree}{{\rm Ree}}
\newcommand{\Sym}{{\rm Sym}}
\newcommand{\Alt}{{\rm Alt}}
\newcommand{\Sp}{{\rm Sp}}
\newcommand{\Q}{{\rm Q}}
\newcommand{\GU}{{\rm GU}}
\newcommand{\PSp}{{\rm PSp}}
\newcommand{\W}{{\rm W}}

\newcommand{\POmega}{{\rm P\Omega}}
\newcommand{\PGammaL}{{\rm P\Gamma L}}
\newcommand{\AGammaL}{{\rm A\Gamma L}}

\newcommand\herm{{\rm H}}
\newcommand\PGammaU{{\rm P\Gamma U}}
\newcommand\PGammaSp{{\rm P\Gamma Sp}}
\newcommand\PGammaO{{\rm P\Gamma O}}

\newcommand\N{\mathbb{N}}

\newcommand{\AS}{{\rm AS}}
\newcommand{\PA}{{\rm PA}}
\newcommand{\TW}{{\rm TW}}
\newcommand{\HA}{{\rm HA}}
\newcommand{\SD}{{\rm SD}}
\newcommand{\HS}{{\rm HS}}

\DeclareMathOperator{\LCM}{LCM}
\DeclareMathOperator{\soc}{soc}

\newcommand{\liso}{\lesssim}
\newcommand\Wr{{\rm \,wr\, }}

\newcommand{\calQ}{\mathcal{Q}}
\newcommand{\calL}{\mathcal{L}}
\newcommand{\calP}{\mathcal{P}}
\newcommand{\calB}{\mathcal{B}}

\newcommand{\stabkernel}[2]{#1_{#2}^{[1]}}

\renewcommand\le{\leqslant}
\renewcommand\ge{\geqslant}

\title[Locally $2$-transitive generalized quadrangles]{A classification of finite locally $2$-transitive\\ generalized quadrangles}

\author{John Bamberg}
\address{Centre for the Mathematics of Symmetry and Computation\\
School of Mathematics and Statistics\\
The University of Western Australia\\
35 Stirling Highway, Crawley, W.A. 6009, Australia.}
\curraddr{}
\email{john.bamberg@uwa.edu.au}

\author{Cai Heng Li}
\address{Department of Mathematics\\ SUSTech International Center for Mathematics\\ Southern University of Science and Technology\\ Shenzhen, Guangdong 518055, P. R. China.}
\curraddr{}
\email{lich@sustech.edu.cn}

\author{Eric Swartz}
\address{Department of Mathematics\\
College of William \& Mary\\
Williamsburg, VA 23187, USA}
\curraddr{}
\email{easwartz@wm.edu}


\begin{document}

\maketitle

\begin{abstract}
Ostrom and Wagner (1959) proved that if the automorphism group $G$ of a finite projective plane $\pi$ acts $2$-transitively
on the points of $\pi$, then $\pi$ is isomorphic to the Desarguesian projective plane and $G$ is isomorphic to $\PGammaL(3,q)$
(for some prime-power $q$). In the more general case of a finite rank $2$ irreducible spherical building, also known as a
\emph{generalized polygon}, the theorem of Fong and Seitz (1973) gave a classification of the \emph{Moufang} examples.  A conjecture of Kantor, made in print in 1991, says that there are only two non-classical examples of flag-transitive generalized quadrangles up to duality.  Recently, the authors made progress toward this conjecture by classifying those finite generalized quadrangles
which have an automorphism group $G$ acting transitively on antiflags. In this paper, we take this classification 
much further by weakening the hypothesis to $G$ being transitive on ordered pairs of collinear points
and ordered pairs of concurrent lines.
\end{abstract}

\section{Introduction}

A \textit{generalized quadrangle} is an incidence geometry of points and lines such that every pair of distinct points determines at most one line and every line contains at least two distinct points,
satisfying the following additional condition, often referred to as the ``GQ Axiom'':
\begin{center}
\begin{tabular}{rp{0.67\textwidth}}
\textsc{GQ Axiom}:& \textit{Given a point $P$ and a line $\ell$ not incident with $P$, there is a unique point on $\ell$ collinear with $P$.}  
\end{tabular}
\end{center}

Generalized quadrangles can alternatively be defined in terms of their \textit{incidence graphs}, which are bipartite and have diameter $4$ and girth $8$.  Indeed, a \textit{generalized $n$-gon} is an incidence geometry whose associated incidence graph has diameter $n$ and girth $2n$.  Generalized $n$-gons were introduced by Jacques Tits \cite{titsngon} in an attempt to characterize families of groups in terms of an associated geometry.  The \textit{classical generalized quadrangles}, which are described in detail in Table \ref{tbl:classical}, are the families related to classical almost simple groups of Lie type, and in each case the points and lines each correspond to totally singular subspaces with respect to a sesquilinear or quadratic form.

In the intervening years, generalized quadrangles have been studied thoroughly from a purely geometric perspective \cite{fgq}, but the connection to group theory remains tantalizing.  One of the outstanding open questions in the area is the classification of \textit{flag-transitive} finite generalized quadrangles, that is, the classification of all finite generalized quadrangles with a group of collineations that is transitive on incident point-line pairs. The following
conjecture was made in print by Kantor \cite{Kantor1991}.

\begin{conj}[W. M. Kantor 1991]\samepage
\label{conj:flag}
 If $\mathcal{Q}$ is a finite flag-transitive generalized quadrangle and $\mathcal{Q}$ is not a classical generalized quadrangle, then (up to duality) $\mathcal{Q}$ is the unique generalized quadrangle of order $(3,5)$ or the generalized quadrangle of order $(15,17)$ arising from the Lunelli-Sce hyperoval.
\end{conj}

Finite generalized polygons satisfying stronger symmetry assumptions, such as the \textit{Moufang condition} \cite{FongSeitz} or \textit{distance-transitivity} \cite{BuekenhoutHvM1994, Ostrom:1959ys}, have been classified.  The current state-of-the-art for generalized quadrangles is the classification of \textit{antiflag-transitive} finite generalized quadrangles in \cite{BLS}, where it was shown that, up to duality, the only non-classical antiflag-transitive generalized quadrangle is the unique generalized quadrangle of order $(3,5)$. (An antiflag is a non-incident point-line pair.)  Notice
that by the GQ Axiom, antiflag-transitivity implies flag-transitivity for a generalized quadrangle.


The aim of this paper is to provide further progress toward Conjecture \ref{conj:flag}.  We study finite generalized quadrangles with a strictly weaker local symmetry condition, \textit{local 2-transitivity}.  
 If $G$ is a subgroup of collineations of a finite generalized quadrangle $\mathcal{Q}$ that is transitive both on pairs of collinear points and pairs of concurrent lines, then $\mathcal{Q}$ is said to be a \textit{locally $(G,2)$-transitive generalized quadrangle}. 
 

Our main result is a complete classification of thick locally $2$-transitive generalized quadrangles, the proof of which relies on the Classification of Finite Simple Groups (CFSG) \cite{CFSG}.

\begin{thm}
 \label{thm:main}
 If $\mathcal{Q}$ is a thick finite locally $(G,2)$-transitive generalized quadrangle and $\mathcal{Q}$ is not a classical generalized quadrangle, then (up to duality) $\mathcal{Q}$ is the unique generalized quadrangle of order $(3,5)$.
\end{thm}

An equivalent definition of a locally 2-transitive generalized quadrangle is that it has an \textit{incidence graph} that is locally 2-arc-transitive; see Section \ref{sect:background} for details.  Since an equivalent definition of an antiflag-transitive generalized quadrangle is that it has a locally $3$-arc-transitive incidence graph, we have moved from classifying locally 3-arc-transitive graphs to classifying locally 2-arc-transitive graphs, a gulf that is generally considered to be wide by those working in the area of graph symmetry.  Indeed, this distinction is borne out in this paper, where the eventual reduction to the point where we are able to check which almost simple groups of Lie type can act primitively on both the points and the lines of such a generalized quadrangle is far more difficult than in the antiflag-transitive case: compare Sections \ref{sect:general}--\ref{sect:prim1} of this paper to \cite[pp. 1558--1560]{BLS}.   

It should also be emphasized that this paper is a sequel to and in many ways an extension of \cite{BLS}.  In particular, the main theorems in Sections \ref{sect:prim1} and \ref{sect:prim2}, which involve the aforementioned checking of almost simple groups of Lie type, rely on the work done in \cite[Sections 5--8]{BLS}.   

After providing necessary background on permutation groups, graph symmetry, and finite generalized quadrangles in Section \ref{sect:background}, the rest of the paper is devoted to the proof of Theorem \ref{thm:main}.  The strategy for proving Theorem \ref{thm:main} is as follows.  Let $\calQ$ be a thick finite locally $(G,2)$-transitive generalized quadrangle.

\begin{itemize}
 \item We prove (Theorem \ref{thm:quasi0poss}) that $G$ must act quasiprimitively on either points or lines.  This is done by an examination of the size of blocks in a system of imprimitivity and relies heavily on Lemma \ref{lem:Brestrictions}.
 \item We prove (Theorem \ref{thm:quasi1}) that, if $G$ is quasiprimitive on points but not on lines, then $\calQ$ is the unique generalized quadrangle of order $(3,5)$.
 \item It is shown (Theorem \ref{thm:main_rephrased}) that, other than for the generalized quadrangle of order $(3,5)$ and its dual, $G$ must be an almost simple group.
 \item Mainly using Proposition \ref{prop:ASst+1}, it is shown in Theorem \ref{thm:prim} that, up to duality, $G$ must be primitive on points.
 \item Proposition \ref{prop:Lie} shows that an almost simple group acting primitively on points must be of Lie type, and, if $T := \soc(G)$ and $P$ is a point, then Theorem \ref{thm:whenlargeTP} shows that $|T| < |T_P|^3$, i.e., $T_P$ is a \textit{large} subgroup of $T$.
 \item We establish (Theorem \ref{thm:PrimImprim35}) that the unique generalized quadrangle of order $(3,5)$ is the only locally $2$-transitive generalized quadrangle with $G$ primitive on points but not lines.  The case when $G$ is almost simple is eliminated by using the characterization of large maximal subgroups of simple groups by Alavi and Burness \cite{largesubs}, slightly modifying the proofs when necessary from \cite[Sections 5--8]{BLS}.
 \item Finally, if $G$ is primitive on both points and lines, then $\calQ$ must be classical (Theorem \ref{thm:PrimPrimClassical}), a result that now follows essentially immediately from \cite[Sections 5--8]{BLS}.
\end{itemize}

\section{Background}
\label{sect:background}

\subsection{Permutation groups}

Let a group $G$ act on a set $\Omega$, and let $\alpha \in \Omega$.  We denote the \textit{orbit} of $\alpha$ under $G$ by $\alpha^G$ and the \textit{stabilizer} of $\alpha$ in $G$ by $G_\alpha$.  Given a set $\Sigma \subseteq \Omega$, we denote by $G_\Sigma$ the setwise stabilizer of $\Sigma$ in $G$, although it should be noted that $G_{\alpha_1\alpha_2 \dots \alpha_n}$ refers to the subgroup of $G$ that fixes each of $\alpha_1, \dots, \alpha_n \in \Omega$ pointwise.  


If a group $G$ acts transitively on a set $\Omega$, then $G$ is said to be \textit{primitive} on $\Omega$ (or \textit{act primitively} on $\Omega$) if $G$ does not preserve any nontrivial partition of $\Omega$.  In such a case that a group $G$ acts transitively on a set $\Omega$ but does preserve a nontrivial partition of $\Omega$, $G$ is said to be \textit{imprimitive} on $\Omega$ (or \textit{act imprimitively} on $\Omega$) and the subsets that make up this nontrivial partition are called \textit{blocks}.  If a group $G$ acts on the set $\Omega$, then $G$ is said to be \textit{quasiprimitive} on $\Omega$ (or \textit{act quasiprimitively} on $\Omega$) if every nontrivial normal subgroup of $G$ is transitive on $\Omega$.  
The so-called \emph{O'Nan-Scott Theorem} characterizes the finite primitive permutation groups according to a division into certain families,
and a similar characterization (by Praeger \cite{quasiprim1}) categorizes the quasiprimitive groups into eight different types.  Six of these types are relevant to this paper, and we describe these six roughly here. In each case, $G$ denotes the finite quasiprimitive group of a particular type, and $\Omega$ denotes the set on which $G$ acts quasiprimitively.  We do not list the full details of each case; only defining characteristics that are used later.  For more details, see \cite[Section 2]{localsarc}.

\begin{description}
 \item[Holomorph Affine (HA)] $G$ has a unique minimal normal subgroup $N \cong C_p^d$ that is elementary abelian, where $p$ is a prime and $d \in \mathbb{N}$.  The set $\Omega$ is identified with $N$ (which is itself identified with the points of the affine space ${\rm AG}(d,p)$), and, if $\alpha \in \Omega$, $G = N : G_\alpha$, where $G_\alpha$ is an irreducible subgroup of $\GL(d,p)$.
 \item[Holomorph Simple (HS)] $G$ is a subgroup of the \textit{holomorph} ${\rm Hol}(T) = T.\Aut(T)$, where $T$ is a finite nonabelian simple group and $G$ contains $T.{\rm Inn}(T)$.  Here, $\Omega = T$ and $G$ has two minimal normal subgroups, both of which are isomorphic to $T$ and act regularly on $\Omega$.
 \item[Almost Simple (AS)] For some finite simple group $T$, $T \le G \le \Aut(T)$, and $T$ is the unique minimal normal subgroup of $G$.  Our only condition on $\Omega$ in this case is that, for $\alpha \in \Omega$, $T \not\le G_\alpha$.  It is possible that $T$ acts regularly on $\Omega$.
 \item[Simple Diagonal (SD)] $G$ has a unique minimal normal subgroup $N \cong T^k$ for some finite nonabelian simple group $T$ and positive integer $k \ge 2$.  Furthermore, for all $\alpha \in \Omega$, $N_\alpha \cong T$, and $G$ transitively permutes the $k$ simple direct factors of $N$.  The set $\Omega$ may be identified with $T^{k-1}$.
 \item[Twisted Wreath (TW)] $G$ has a unique minimal normal subgroup $N \cong T^k$, where $T$ is a finite nonabelian simple group and $k \ge 2$ is a positive integer, and the set $\Omega$ is identified with $N$.
 \item[Product Action (PA)] $G$ has a unique minimal normal subgroup $N \cong T^k$ for some finite nonabelian simple group $T$ and positive integer $k \ge 2$.  For $\alpha \in \Omega$, $N_\alpha \neq 1$ and is not isomorphic to $T^n$ for any $n \le k$.  Moreover, $G \le \Aut(T) \Wr S_k$, $G$ acts transitively by conjugation on the simple direct factors of $N$, the group $G$ preserves a product structure $\Delta^k$ on $\Omega$, and, if $\alpha \in B \subseteq \Omega$, where $B = (\delta, \delta, \dots, \delta)$ for $\delta \in \Delta$, $N_\alpha$ is a subdirect subgroup of the stabilizer $N_B \cong T_\delta^k$, i.e., $N_\alpha$ projects onto $T_\delta$ in each coordinate.  
\end{description}

\subsection{Graph symmetry}

Let $\Gamma$ be a graph with vertex set $V$.  
If $\alpha \in V$, then the set of all neighbors of $\alpha$, i.e., the set of all vertices adjacent to $\alpha$ in $\Gamma$, is denoted by $\Gamma(\alpha)$, and, if $G \le \Aut(\Gamma)$, then the permutation group induced by $G_\alpha$ on $\Gamma(\alpha)$ is denoted by $G_\alpha^{\Gamma(\alpha)}$.  The distance between two vertices $\alpha, \beta$ is the number of edges in a shortest path from $\alpha$ to $\beta$ and is denoted by $d(\alpha, \beta)$. The group 
\[\stabkernel{G}{\alpha} := \{ g \in G: \beta^g = \beta \text{ for all } \beta \text{ such that } d(\alpha, \beta) \le 1\}\] 
is a normal subgroup of $G_\alpha$ and is referred to as the \textit{kernel of the local action}, since $G_\alpha^{\Gamma(\alpha)} \cong G_\alpha/\stabkernel{G}{\alpha}$.


An \textit{$s$-arc} of a graph is a sequence of $s+1$ vertices $(\alpha_0, \alpha_1, \dots, \alpha_s)$ such that, when $0 \le i \le s-1$, $\alpha_i$ is adjacent to $\alpha_{i+1}$, and, when $1 \le i \le s-1$, $\alpha_{i-1} \neq \alpha_{i+1}$.  Note that repeated vertices are allowed as long as there are no returns.  A subgroup $G \le \Aut(\Gamma)$ is said to be \textit{locally $(G,s)$-arc-transitive} if $\Gamma$ contains an $s$-arc and, for any vertex $\alpha$, $G$ is transitive on the $s$-arcs starting at $\alpha$.  When such a $G$ exists, $\Gamma$ is said to be \textit{locally $s$-arc-transitive}.  While it is possible for a graph $\Gamma$ to be locally $(G,s)$-arc-transitive while $G$ is intransitive on $V$, in this case the group $G$ would be transitive on the edges of $\Gamma$, and hence $G$ would have exactly two orbits on $V$.  The following lemma provides a characterization of locally $(G,2)$-arc-transitive graphs in terms of the action of the stabilizer of a vertex on its neighbors.

\begin{lem}[{\cite[Lemma 3.2]{localsarc}}]
\label{lem:loc2trans}
 Let $\Gamma$ be a graph such that all vertices have valency at least two.  Then, $\Gamma$ is locally $(G,2)$-arc-transitive if and only if for every vertex $\alpha$, $G_\alpha^{\Gamma(\alpha)}$ is a $2$-transitive permutation group, i.e., if and only if $G_\alpha$ acts $2$-transitively on $\Gamma(\alpha)$.
\end{lem}

A graph $\Gamma$ with group of automorphisms $G$ is said to be \textit{$G$-locally primitive} if for any $\alpha \in V$ the induced action of $G_\alpha$ on the neighbors of $\alpha$ is primitive, that is, if $G_\alpha^{\Gamma(\alpha)}$ is primitive on $\Gamma(\alpha)$.  When such a $G$ exists, $\Gamma$ is said to be \textit{locally primitive}. 

If a group $G \le \Aut(\Gamma)$ has a normal subgroup $N$ that is intransitive on $V$, then we define the \textit{(normal) quotient graph} $\Gamma_N$ to have vertex set the $N$-orbits of $V$, where two $N$-orbits $\Sigma_1$ and $\Sigma_2$ are adjacent in $\Gamma_N$ if and only there exist vertices $\alpha \in \Sigma_1$ and $\beta \in \Sigma_2$ such that $\alpha$ is adjacent to $\beta$ in $\Gamma$.  By a result of Giudici, Li, and Praeger \cite{localsarc}, if $\Gamma$ is connected and locally $(G,s)$-arc-transitive, then either $\Gamma_N$ is locally $(G/N,s)$-arc-transitive or $\Gamma_N$ is a  complete bipartite graph $K_{1,n}$.  This demonstrates the importance of studying locally $s$-arc-transitive graphs with a group of automorphisms acting quasiprimitively on at least one orbit of vertices, and the following results characterize the types of possible quasiprimitive actions of $G$ on the orbits of vertices of a locally $(G,s)$-arc-transitive graph $\Gamma$.

\begin{lem}[{\cite[Theorem 1.3]{localsarc}}]
\label{lem:types1}
 Let $\Gamma$ be a finite locally $(G,s)$-arc-transitive connected graph with $s \ge 2$ such that $G$ acts faithfully on both its orbits $\Delta_1$ and $\Delta_2$ of vertices but only acts quasiprimitively on $\Delta_1$.  Then the quasiprimitive type of $G$ on $\Delta_1$ is  of type $\HA$, $\HS$, $\AS$, $\PA$, or $\TW$.
\end{lem}

\begin{lem}[{\cite[Theorem 1.2]{localsarc}}]
\label{lem:possactions}
Let $\Gamma$ be a connected locally $(G,2)$-arc-transitive graph such that $G$ has two orbits on vertices and $G$ acts faithfully and quasiprimitively on both orbits with type $\{X,Y\}.$  Then either $X = Y \in \{\HA, \TW, \AS, \PA \}$ or $\{X,Y\} = \{\SD, \PA\}$. 
\end{lem}

The graph $\Gamma$ is said to be a \textit{cover} of $\Gamma_N$ if $|\Gamma(\alpha) \cap \Sigma_2| = 1$ for each edge $\{\Sigma_1, \Sigma_2\}$ in $\Gamma_N$ and $\alpha \in \Sigma_1$.  The following lemma provides a characterization of the case when $\Gamma$ is a locally $(G,s)$-arc-transitive bipartite graph such that $G$ contains a nontrivial normal subgroup that is intransitive on each bipart.

\begin{lem}[{\cite[Lemma 5.1]{localsarc}}]
\label{lem:bipintrans}
 Let $\Gamma$ be a connected $G$-locally primitive bipartite graph with $G$-orbits $\Delta_1$ and $\Delta_2$ on $V(\Gamma)$ such that $|\Delta_i| > 1$ for each $i$.  Suppose that there exists $N \lhd G$ such that $N$ is intransitive on both $\Delta_1$ and $\Delta_2$.  Then
 \begin{enumerate}[(i)]
  \item $\Gamma$ is a cover of $\Gamma_N$.
  \item $N$ acts semiregularly on $V(\Gamma)$ and $G^{V_N} \cong G/N$, where $V_N$ is the vertex set of $\Gamma_N$.
  \item $\Gamma_N$ is $(G/N)$-locally primitive.  Furthermore, if $\Gamma$ is locally $(G,s)$-arc-transitive, then $\Gamma_N$ is locally $(G/N, s)$-arc-transitive.
 \end{enumerate}
\end{lem}

On the other hand, the following lemma provides information in the case when every nontrivial normal subgroup of automorphisms is transitive on at least one orbit of vertices.

\begin{lem}{\cite[Lemma 5.4]{localsarc}}
\label{lem:biptrans}
 Let $\Gamma$ be a finite connected graph such that $G \le \Aut(\Gamma)$ has two orbits $\Delta_1$ and $\Delta_2$ on vertices and $G$ acts faithfully on both orbits.  Suppose that every nontrivial normal subgroup $N$ of $G$ is transitive on at least one of the $\Delta_i$.  Then, $G$ acts quasiprimitively on at least one of its orbits.
\end{lem}

Finally, the following technical but useful lemma provides information about the composition factors of the vertex stabilizer $G_\alpha$.  

\begin{lem}[{\cite[Theorem 1.1]{Song}}]\label{comp-factor}
Let $\Gamma$ be a connected graph, and let $G\le\Aut(\Gamma)$ be transitive on the edge set.
Let $\{\alpha,\beta\}$ be an edge of $\Gamma$.
Then each composition factor of $G_\alpha$ is a composition factor of $G_\alpha^{\Gamma(\alpha)}$, $G_{\alpha\beta}^{\Gamma(\beta)}$, or $G_{\alpha\beta}^{\Gamma(\alpha)}$.
Moreover, if $|\Gamma(\alpha)|\geqslant|\Gamma(\beta)|$, then $|\Gamma(\alpha)|$ is not smaller than the smallest permutation degree
of any composition factor of $G_\alpha$.
\end{lem}

\subsection{Finite generalized quadrangles}

Let $\mathcal{Q}$ be a finite generalized quadrangle with point set $\mathcal{P}$ and line set $\mathcal{L}$.  If each point has at least three lines passing through it, and each line contains at least three points, then a simple combinatorial argument shows that each line is incident with a constant number of points, and each point is incident with a constant number of lines. The generalized quadrangle $\mathcal{Q}$ is said to have \textit{order} $(s,t)$ if each line is incident with $s+1$ points and each point is incident with $t+1$ lines.  If both $s$ and $t$ are at least $2$, then the generalized quadrangle is said to be \textit{thick}.  The following omnibus lemma details how $(s,t)$ determines the total number of points and lines and how the two parameters are constrained.

\begin{lem}[{\cite[1.2.1, 1.2.2, 1.2.3, 1.2.5]{fgq}}]
\label{lem:GQbasics}
Let $\mathcal{Q}$ be a finite generalized quadrangle of order $(s,t)$, where $s,t > 1$.  Then the following hold:
\begin{itemize}
 \item[(i)] $|\mathcal{P}| = (s+1)(st + 1)$ and $|\mathcal{L}| = (t+1)(st+1)$;
 \item[(ii)] $s + t$ divides $st(s+1)(t+1)$;
 \item[(iii)] $t \leqslant s^2$ and $s \leqslant t^2$;
 \item[(iv)] if $s < t^2$, then $s \leqslant t^2 - t$, and if $t < s^2$, then $t \leqslant s^2 - s$.
\end{itemize}
\end{lem}

%
%

Given $P, Q \in \mathcal{P}$, we write $P \sim Q$ if $P$ and $Q$ are collinear, and, similarly, if $\ell, m \in \mathcal{L}$, we write $\ell \sim m$ if $\ell$ and $m$ are concurrent.  For a set of points $\Delta \subseteq \mathcal{P}$,
\[ \Delta^\perp := \{ P \in \mathcal{P} : P \sim Q \text{ for all } Q \in \Delta\}.\]
By convention, $P \in P^\perp$.

The \textit{collineation group} or \textit{automorphism group} of a generalized quadrangle is the group of all permutations of points that preserve collinearity and non-collinearity.  We now include some results from \cite{BLS} (and a related result) that are used repeatedly later in the paper.    

\begin{lem}[{\cite[Lemma 2.2]{BLS}}]
\label{lem:ratio} Let $G$ be a group of automorphisms of a generalized quadrangle $\calQ$ that is transitive on both the point set $\mathcal{P}$ and the line set $\mathcal{L}$.
Then, for $P\in\calP$ and $\ell\in\calL$, $$\frac{s+1}{t+1} = \frac{|G_\ell|}{|G_P|}.$$
\end{lem} 

\begin{lem}
\label{lem:Pbounds} 
Assuming $s \le t$, the following inequalities hold:
\begin{itemize}
 \item[(i)] \cite[Lemma 2.3]{BLS} $(t+1)^2 < |\mathcal{P}| < (t+1)^3$;
 \item[(ii)] \cite[Lemma 2.3]{BLS} $s^2 (t+1) < |\mathcal{P}| < s(t+1)^2$.
\end{itemize}
Without any assumption about $s$ and $t$, the following inequality holds:
\begin{itemize}
 \item[(iii)] $|\calP| < (t+1)^5$. 
\end{itemize}
\end{lem}

\begin{proof}
 To prove (iii), we note that $s \le t^2$ by Lemma \ref{lem:GQbasics}(iii), and so
 \begin{equation*}
|\calP| = (s+1)(st+1) \le (t^2 + 1)(t^3 + 1) < (t+1)^5. \qedhere  
 \end{equation*}
\end{proof}

If $H_1, H_2, \dots, H_r$ are permutation groups on sets $\Omega_1, \Omega_2, \dots, \Omega_r$, respectively, then the \textit{product action} of the direct product $H_1 \times H_2 \times \cdots \times H_r$ on the Cartesian product $\Omega_1 \times \Omega_2 \times \cdots \times \Omega_r$ is the action $(\omega_1, \dots, \omega_r)^{(h_1, \dots, h_r)} = (\omega_1^{h_1}, \dots, \omega_r^{h_r}).$
The following result provides useful information in the event that a group of automorphisms of a generalized quadrangle has the product action on the point set.

\begin{lem}[{\cite[Theorem 1.2]{BPP2}}]
\label{lem:PAdimrest}
 Let $\Omega_1, \dots, \Omega_r$ be finite sets with $2 \le |\Omega_1| \le \cdots \le |\Omega_r|$, where $r \ge 1$, and let $H_i \le \Sym(\Omega_i)$ for each $i \in \{1, \dots, r\}$.  Assume further that $H_1$ is nontrivial and that its action on $\Omega_1$ is not semiregular.  Suppose that $N = H_1 \times H_2 \times \cdots \times H_r$ is a collineation group of a thick finite generalized quadrangle $\calQ = (\calP, \calL)$ of order nto equal to $(2,4)$, such that $\calP = \prod_{i = 1}^r \Omega_i$ and $N$ has the product action on $\calP$.  Then $r \le 4$, and every nonidentity element of $H_1$ fixes less than $|\Omega_1|^{1 - r/5}$ points of $\Omega_1$. 
\end{lem}

For a given generalized quadrangle $\mathcal{Q}$, the \textit{incidence graph} $\Gamma$ of $\mathcal{Q}$ is the graph with vertex set $\mathcal{P} \cup \mathcal{L}$ with edges between incident points and lines.  The collineation group of a generalized quadrangle can thus be identified with the automorphism group of the incidence graph.  A \textit{locally $2$-transitive generalized quadrangle} is a finite generalized quadrangle with a locally $2$-arc-transitive incidence graph.  From the definition and the above discussion of locally $2$-arc-transitive graphs, it is clear that locally $2$-transitive generalized quadrangles are equivalently those finite generalized quadrangles with a collineation group that is transitive both on ordered pairs of collinear points and on ordered pairs of concurrent lines.  We now describe the known locally $2$-transitive generalized quadrangles. 

Up to duality, apart from one example, all known locally $2$-transitive generalized quadrangles are \textit{classical generalized quadrangles}, which are the generalized quadrangles associated with classical groups of Lie type.  The following table summarizes the information about the classical generalized quadrangles.  The notation $E_q^a$ (and sometimes just $q^a$) denotes an elementary abelian group of order $q^a$, where $q$ is a prime power, and the notation $E_q^{a+b}$ denotes a special group order $q^{a+b}$ with center of order $q^a$.  The column ``$G$'' refers to the full collineation group of $\mathcal{Q}$.

\begin{table}[ht]
\caption{The classical generalized quadrangles.}\label{tbl:classical}
\begin{tabular}{l|c|c|c|l}
\toprule
$\mathcal{Q}$& Order & $G$ & $\soc(G)$ & Point stabilizer in $\soc(G)$\\
\midrule
$\W(3,q)$, $q$ odd& $(q,q)$ & $\PGammaSp(4,q)$ &$\PSp(4,q)$ & $E_q^{1+2}: ( \GL(1,q)\circ \Sp_2(q))$\\
$\W(3,q)$, $q$ even& $(q,q)$ & $\PGammaSp(4,q)$ &$\Sp(4,q)$ & $E_q^{3}:\GL(2,q)$\\
$\Q(4,q)$, $q$ odd & $(q,q)$ & $\PGammaO(5,q)$ &$\POmega(5,q)$ & $E_q^{3}:(  (\frac{(q-1)}{2}\times \Omega(3,q)).2)$\\
$\Q^-(5,q)$ & $(q,q^2)$& $\PGammaO^-(6,q)$ &$\POmega^-(6,q)$ & $E_q^4: (\frac{q-1}{|Z(\Omega^-(6,q))|}\times \Omega^-(4,q))$\\
$\herm(3,q^2)$ & $(q^2,q)$& $\PGammaU(4,q)$ &$\PSU(4,q)$ & $E_q^{1+4}:\left(\SU(2,q):\frac{q^2-1}{\gcd(q+1,4)}\right)$\\
$\herm(4,q^2)$ &$(q^2,q^3)$& $\PGammaU(5,q)$ &$\PSU(5,q)$ & $E_q^{1+6}:\left(\SU(3,q):\frac{q^2-1}{\gcd(q+1,5)}\right)$\\
$\herm(4,q^2)^D$ &$(q^3,q^2)$& $\PGammaU(5,q)$ &$\PSU(5,q)$ & $E_q^{4+4}:\GL(2,q^2)$\\
\bottomrule
\end{tabular}
\end{table}

The only other known locally $2$-transitive generalized quadrangle up to duality is the unique generalized quadrangle of order $(3,5)$. The full collineation group of this generalized quadrangle is isomorphic to $2^6{:}(3.A_6.2)$,  the
stabilizer of a point is isomorphic to $3.A_6.2$, and the stabilizer of a line is isomorphic to
$(A_5 \times A_4).2$. (See \cite[\S2.2]{BLS}).

Finally, we conclude this section with the following lemma, which appeared in an earlier version of this paper and the authors believe is interesting in its own right.

\begin{lem}
\label{lem:s=t}
Let $\mathcal{Q}$ be a finite thick generalized quadrangle with $s+1 = p^a$ and $t+1 = p^b$, where $p$ is a prime.  Then $s = t$.   
\end{lem}

\begin{proof}
Without a loss of generality we may assume that $s \leqslant t.$  First, consider the case when $p = 2$.  By Lemma \ref{lem:GQbasics}(ii), we have that $s+t$ divides $st(s+1)(t+1)$.  Noting that $(s+t)/2 = 2^{b-1} + 2^{a-1} - 1$ and $2^{b-1} \geqslant 2^{a-1} > 1$ since $s>1$, $(s+t)/2$ is odd and hence must divide $st$.  Thus $s+t$ must divide $2st$, and since $2st = 2s(s+t) - 2s^2$, $s+t$ must divide $2s^2$ as well.  This in turn implies that $s+t$ must divide $2st - 2s^2 = 2s(t-s)$, i.e., that $2^b + 2^a - 2$ must divide $2(2^a - 1)((2^b - 1) - (2^a - 1)) = 2^{a+1}(2^a - 1)(2^{b-a} - 1)$.  Again, $(s+t)/2$ is odd, so $(s+t)/2$ must divide $(2^a - 1)(2^{b-a} - 1)$, i.e., $2^{b-1} + 2^{a-1} - 1$ must divide $2^b - 2^a - 2^{b-a} +1$.  Since both $(s+t)/2$ and $(2^a - 1)(2^{b-a}-1)$ are nonnegative and $2\cdot (s+t)/2 = 2^b + 2^a - 2 > 2^b - 2^{b-a} - 2^a + 1$, this leaves two cases: either $2^{b-1} + 2^{a-1} - 1 = 2^b - 2^{b-a} - 2^a + 1$ or $(2^a - 1)(2^{b-a} - 1) = 2^b - 2^{b-a} - 2^a + 1 = 0.$  
In the first case, we rearrange terms to see that $3\cdot2^{a-1} - 2 = 2^{b-1} - 2^{b-a}$.  Now, $b-a \leqslant a - 1$; otherwise, $b \geqslant 2a$ and $(t+1) \geqslant (s+1)^2$, which means $t \geqslant s^2 + 2s > s^2$, a contradiction to Lemma \ref{lem:GQbasics}(iii). This means that $-2 \equiv 0 \pmod {2^{b-a}}$, which implies that $2^{b-a}$ is $1$ or $2$.  If $2^{b-a}$ is $1$, we are done.  If $2^{b-a} = 2$, then $2^b = 2^{a+1}$, and we have $3\cdot2^{a-1} - 2 = 2^a - 2$, which has no integral solutions.
In the second case, we have $(2^a - 1)(2^{b-a} - 1) = 0$.  Since $s > 1$, we must have $2^{b-a} - 1 = 0$.  Therefore, $2^a = 2^b$ and $s = t$, as desired. 

We continue with the assumption that $s \leqslant t$, and we now assume that $p \geqslant 3$ is an odd prime.  Thus $s+t = p^a + p^b - 2$ is coprime to $p$, and by Lemma \ref{lem:GQbasics}(ii), $s+t$ divides $st(s+1)(t+1) = p^{a+b}(p^a - 1)(p^b - 1)$.  Hence $p^a + p ^b - 2$ divides $(p^a - 1)(p^b - 1)$, i.e., $s+t$ divides $st$.  Clearly, $s+t$ divides $(s+t)^2$, so $s+t$ divides $(s+t)^2 - 4st = (s - t)^2 = p^{2a}(p^{b-a} - 1)^2$, i.e., $p^b + p^a - 2$ divides $(p^{b-a} - 1)^2 = p^{2b - 2a} - 2p^{b-a} + 1$.  However, we know that $t \leqslant s^2$ by Lemma \ref{lem:GQbasics}(iii), so $p^b < p^{2a}$, and $(p^{b-a} - 1)^2 < p^{2b - 2a} < p^b < p^b + p^a - 2$.  Hence $p^{b-a} - 1 = 0$ and $a = b$, as desired.  
\end{proof}

\section{Reduction to quasiprimitivity} 
\label{sect:general}

Let $\mathcal{Q}$ be a finite generalized quadrangle of order $(s,t)$, where $s,t > 1$.  Let $\Gamma$ be the associated incidence graph, and suppose that there exists a subgroup $G \leqslant \Aut(\Gamma)$ such that $\Gamma$ is locally $(G,2)$-arc-transitive.  We will denote by $\mathcal{P}$ the points of $\mathcal{Q}$ and by $\mathcal{L}$ the lines of $\mathcal{Q}.$  Throughout, we will abuse notation a bit and use $\mathcal{P}$ and $\mathcal{L}$ to refer to the biparts of $\Gamma$ as well.    

\begin{lem}
 \label{lem:Brestrictions}
 Let $G$ preserve a nontrivial system of imprimitivity on $\mathcal{P}$ and let $B$ be a block of points containing a point $P$.  Then,  $B \cap P^\perp = \{P\}$ and $|B| = bs + 1$, where $b$ is the number of points in $B$ (other than $P$) that are collinear with a given point $P^\prime \in P^\perp \backslash \{P\}$.  Moreover, if $s < t$, then $|B| = st + 1$, and  if $s \ge t$, then either $|B| = st + 1$ or $t \mid s$ and $|B| = s + 1$. 
\end{lem}

\begin{proof}
 Let $P$ be a fixed point of $\mathcal{Q}$ and consider the block $B$ containing $P$.  If $Q \in B$ and $Q \neq P$, then, by the local primitivity of the incidence graph $\Gamma$, we have $Q \not\sim P$.  Since $\Gamma$ is locally $(G,2)$-arc-transitive, $G_P$ is transitive on $P^\perp \backslash \{P\}$.  Moreover, $G_P$ fixes $B$, and so each point $P^\prime \in P^\perp \backslash \{P\}$ is collinear with the same number of points (other than $P$) in $B$, say $b$.  We will count the number of pairs $(P^\prime, Q)$, where $P^\prime \in P^\perp \backslash \{P\}$, $Q \in B \backslash \{P\}$, and $P^\prime \sim Q$, in two different ways.  On the one hand, there are $b$ choices for $Q$ for each choice of $P^\prime$, so the number of pairs is
 \[ |P^\perp \backslash \{P\}| b = s(t+1)b.\]
 On the other hand, for each $Q \in B$, $Q \neq P$, $|\{P,Q\}^\perp| = t+1$ by the GQ Axiom, so the number of pairs is also $|B \backslash \{P\}|(t+1)$, and hence $|B \backslash \{P\}| = s b$; that is,
 \begin{equation}\label{eq:bsplus1} 
 |B| = bs + 1,
 \end{equation}
 as desired.

 If there are $n$ blocks, then
 \[bsn + n = n|B| = |\mathcal{P}| = (s+1)(st + 1) = s(st+t +1) + 1,\]
 and so $n \equiv 1 \pmod s$.  Let $n = cs + 1$, where $c \in \N$.  Then
 \[(cs + 1)(bs+1) = (s+1)(st + 1),\]
 and so 
 \begin{equation}\label{eq:main_equality}
 s\left(s(t-bc) + (t+1) - (b+c) \right) = 0.
 \end{equation}
 Since $s \neq 0$, we have 
 \[ s(t - bc) + (t+1) - (b+c) = 0.\]
 Suppose $b + c \ge t + 2$.  Since $b,c \in \N$, $bc \ge t + 1$.  However, this means
 \[ s(t - bc) + (t+1) - (b+c) < 0,\]
(by Equation \eqref{eq:main_equality}) a contradiction.  Hence 
\begin{equation}\label{eq:sbct}
b + c \le t + 1,\quad bc \ge t. 
 \end{equation}
Recall from above that given any $P^\prime \in P^\perp \backslash \{P\}$, there are exactly $b$ points other than $P$ in $B$ collinear with $P^\prime$. 
Let $C$ be a block of points containing points collinear with $P$.  Since $b > 0$, this means that there are exactly $b + 1$ points in $C$ collinear with $P$.  

If $b=t$, then $|B| = st + 1$ (by Equation \eqref{eq:bsplus1}) and our conclusion holds.  
So assume from now on that $b<t$.
This means that not all lines incident with $P$ are incident with a point in $C$.  Now, since $b > 0$, two of the $t + 1$ lines incident with $P$ are incident with points in $C$.  Since $G_P$ is $2$-transitive on incident lines, this means that, for every pair $\{\ell, \ell^\prime\}$ of lines incident with $P$, there are exactly $x$ blocks (other than the block containing $P$) containing a point incident with $\ell$ and a point incident with $\ell^\prime$, where $x$ is some positive integer.  Thus, there are exactly $x {{t+1} \choose 2}$ pairs of points $\{P^\prime, P^{\prime\prime}\}$ that are both incident with $P$ and contained in the same block.  On the other hand, since a block $C$ containing points collinear with $P$ contains exactly $b+1$ points collinear with $P$ and there are $t(s+1)$ points in total collinear with $P$, there are $t(s+1)/(b + 1)$ such blocks and ${{b+1} \choose 2}$ choices within each block, i.e., we have
\[ x \cdot {{t+1} \choose 2} = \frac{(t+1)s}{b+1} \cdot {{b+1} \choose 2},\]
which implies that $xt = bs$.  Hence 
\begin{equation}\label{eq:bsxt}
|B| = bs + 1 = xt + 1, 
\end{equation}
and $|B| \equiv 1 \pmod {\LCM(s,t)}$.

We now consider two cases.  Suppose first that $s < t$. (We want to show that $|B| = st + 1$.)
Now, the number of blocks of points, $n$, satisfies \[ n(xt + 1) = (s + 1)(st + 1),\] and so $n \equiv {(s+1)} \pmod t$.  Let $n = jt + (s+1)$, where $j \in \N \cup \{0\}$.  Thus,
 \[ xjt^2 + x(s+1)t + jt + (s+1) = (xt + 1)n = s^2t + st + (s + 1),\]
 and so
 \[xjt + x(s+1) + j = s^2 + s\]
 and
 \[j(bs + 1) = j(xt + 1) = (s+1)(s-x).\]
 Now, this shows that $j \equiv -x \pmod s,$ and so we let $j = sk - x$ for some $k \in \N$, i.e.,
 \[(sk - x)(bs + 1) = (s-x)(s+1).\]
 Now, if $k > 1$, then $sk - x > s - x$, a contradiction since $bs+1 < s+1$.  Hence $j = s -x$.  There are now two cases.  First, if $s - x \neq 0$, then $bs+1 = s+1$ and so $b = 1$.  However, since $t \mid bs$, we have $s = t$, a contradiction to $s<t$.  Hence it must be that $x = s$, which implies that $|B| = st + 1$, as desired.
 
Suppose now that $s \ge t$. (We want to show that either $|B| = st + 1$, or $t \mid s$ and $|B| = s + 1$.) Also suppose that $b+c < t + 1$.  Then, by Equation \eqref{eq:main_equality},
 \[ s(bc - t) = (t+1) - (b+c) > 0,\]
  and so 
 \[s \le s(bc - t) = (t+1) - (b+c) < t,\]
 a contradiction to $s \ge t$.  Thus $b+c = t+1$, which implies $bc = t$. Then, by Equation \eqref{eq:sbct}, $t+1 \ge b + t/b$, showing $\{b,c\} = \{1,t\}$. If $b = t$, then $|B| = st + 1$, and we are done.  Otherwise, $b<t$,
 we have $b=1$ and $c=t$. Thus, by Equation \eqref{eq:bsxt}, 
 $xt + 1 = bs + 1 = s + 1$ and $t \mid s$, as desired.
\end{proof}  

\begin{lem}
 \label{lem:blocks=t}
 Let $\calQ$ be a thick locally $(G,2)$-transitive generalized quadrangle.  If $s=t$ and $G$ preserves systems of imprimitivity on both $\mathcal{P}$ and $\mathcal{L}$, then, up to duality, each block of lines has size $s^2 + 1$.  
\end{lem}

\begin{proof}
Let $P$ be a point of $\mathcal{Q}$ contained in a block $B$, and assume $|B| \neq s^2+1$.  By Lemma \ref{lem:Brestrictions}, $|B| = s + 1$.  We also know from Lemma \ref{lem:Brestrictions} that 
 \[s+1 = |B| = bs + 1,\]
 where $b$ denotes the number of points of $B \backslash \{P\}$ with which a point $P^\prime \in P^\perp \backslash \{P\}$ is collinear.  In other words, each point of $P^\perp \backslash \{P\}$ is collinear with a unique point in $B \backslash \{P\}$.  Moreover, if $D$ is a block of lines, and we assume $|D| \neq s^2 + 1$, then analogous results hold for $D$ as well; in particular, $|D| = s + 1$.
 
Now, there are exactly $s^2 + 1$ distinct blocks of points.  Again, let $P$ be a fixed point of $\mathcal{Q}$.  Since 
\[ (s^2 + 1) - 1 = s^2 < s^2 + s = |P^\perp \backslash \{P\}|,\]
there is some block that contains more than one point of $P^\perp \backslash \{P\}$.  Since $G_P$ is transitive on collinear points, each block containing points collinear with $P$ contains exactly the same number of points collinear with $P$.  On the other hand, we saw above that each point in $P^\perp \backslash \{P\}$ is collinear with exactly two points in $B$; hence, every block of points (other than $B$) contains either zero or two points collinear with $P$.  

This means that there are exactly $s(s+1)/2$ blocks that nontrivially intersect $P^\perp \backslash \{P\}$ in exactly two points.  Consider one block $C$ that nontrivially intersects $P^\perp \backslash \{P\}$.  By the local primitivity of a line stabilizer (Lemma \ref{lem:loc2trans}), the two points $Q,R \in C$ collinear with $P$ are not themselves collinear.  Thus there exists a pair of distinct lines $\ell, \ell^\prime$ incident with $P$ such that $Q$ is incident with $\ell$ and $R$ is incident with $\ell^\prime$.  By Lemma \ref{lem:loc2trans}, $G_P$ is $2$-transitive on the lines incident with $P$, so there is such a block for every pair of lines.  Since there are exactly ${{s+1} \choose 2}$ distinct blocks that nontrivially intersect $P^\perp \backslash \{P\}$, the blocks that meet $P^\perp \backslash \{P\}$ are in one-to-one correspondence with pairs of lines.  Moreover, let $\ell$ be a fixed line collinear with $P$.  Then, each point of $B \backslash \{P\}$ is collinear with a unique point of $\ell$.  That is, since there are $s+1$ points incident with $\ell$ and $|B| = s+1$, there is a one-to-one correspondence between the points of $B \backslash \{P\}$ and points incident with $\ell$ other than $P$.  This implies that $G_{P\ell}$ is transitive on $B \backslash \{P\}$.  

%

Let $\mathcal{B}$ be the collection of blocks of points and $\mathcal{D}$ be the collection of blocks of lines.  Since $G_{P \ell}$ is transitive on $B \backslash \{P\}$ and $G_P$ is primitive on lines incident with $P$, if $D$ is the block containing the line $\ell$, then either each vertex of $B \backslash \{P\}$ is adjacent to exactly one vertex of $D$ or each vertex of $B$ is adjacent to no vertices in $D$.  Suppose first that each vertex of $B \backslash \{P\}$ is adjacent to no vertices of $D$.  This means that there is only a single edge between the blocks $B$ and $D$, and so, by flag-transitivity, there must be $(s+1)^2$ blocks of lines containing a line incident with a point in $B$.  However, $(s+1)^2 > |\mathcal{D}| = s^2 + 1$, a contradiction.  Hence each vertex of $B$ is adjacent to exactly one vertex of $D$.  

As noted above, given a block of points $B$ and a block of lines $D$, if a single point in $B$ is incident with a line in $D$, then each point of $B$ is incident with exactly one line of $D$.  In particular, if $\Gamma$ is the incidence graph of $\calQ$, then the induced subgraph $\Gamma[B \cup D]$ is a perfect matching.  Thus, if $\Delta$ is defined to be the quotient graph with vertex set $\mathcal{B} \cup \mathcal{D}$, where $B \sim D$ if and only if $P \sim \ell$ for some $P \in B$, $\ell \in D$, then $\Gamma$ is a regular cover of $\Delta$.

Let $\Delta_2$ be the \textit{distance 2 graph} of $\Delta$: the vertex set of $\Delta_2$ is the vertex set of $\Delta$, and two vertices are adjacent in $\Delta_2$ if and only if they are at distance two in $\Delta$.  In a generalized quadrangle, points and lines are always either at distance one or three, whereas distinct points (respectively, distinct lines) are always at distance two or four.  Hence, $\Delta_2$ contains exactly two components of size $s^2 + 1$.  Moreover, as noted above, given a block $B$ of points with $P \in B$, there are exactly ${{s+1} \choose 2}$ distinct blocks of points that nontrivially intersect $P^\perp \backslash \{P\}$, and hence $B$ (and thus every vertex) has degree exactly ${{s+1} \choose 2}$ in $\Delta_2$.  Moreover,  since the blocks that meet $P^\perp \backslash \{P\}$ are in one-to-one corresponce with pairs of lines, $\Delta_2$ is a \textit{locally triangular}, and, in particular, \textit{locally} $T(s+1)$ graph: the neighborhood of a vertex is isomorphic to the triangular graph $T(s+1)$, which itself has vertices corresponding to the ${{s+1} \choose 2}$ distinct $2$-subsets of $\{1, \dots, s+1\}$, and two vertices are adjacent if and only if the $2$-subsets intersect in exactly one element.  (See \cite{DRG} for further information about locally triangular graphs.)  

Since each component of $\Delta_2$ is locally triangular, by \cite[Proposition 4.3.9]{DRG}, each component of $\Delta_2$ is a distance $2$ graph of a bipartite \textit{rectagraph} with $c_3 = 3$: a rectagraph is a connected graph where any two neighbors have no common neighbors and any two nonneighbors have exactly two common neighbors, and that $c_3 = 3$ means for every two vertices $\alpha$ and $\beta$ at distance three, there are exactly $3$ neighbors of $\alpha$ at distance two from $\beta$.  Moreover, since $\Delta$ is a bipartite graph, every two blocks at distance two are in the same bipart, whereas two neighboring blocks are in different biparts.  This implies that, given two blocks $C_1$ and $C_2$ at distance two, there are no blocks that are both neighbors of $C_1$ and at distance two from $C_2$.  In the terminology of distance-regular graphs, this means that $a_2 = 0$ for $\Delta$.  By \cite[Lemma 4.3.5]{DRG}, any rectagraph with $c_3 = 3$ and $a_2 = 0$ has the property that any $3$-claw, i.e., induced subgraph isomorphic to $K_{1,3}$, determines a unique $3$-cube, and, by \cite[Proposition 4.3.8]{DRG}, any rectagraph with the property that any $3$-claw determines a unique $3$-cube has $2^n$ vertices for some $n \in \N$.  Hence, $\Delta$ has $2^n$ vertices for some $n \in \N$. In particular, since there is more than one block in each bipart, each bipart of $\Delta$ has $2^a$ vertices for some $a \in \N$.  However, we already know that each bipart has exactly $s^2 + 1$ vertices, so $s^2 + 1 = 2^a$ for some $a \in \N$.  We are also assuming that $s \ge 2$, so this means
\[s^2 \equiv 3 \pmod 4,\]
a contradiction.  Hence $|D| = s^2 + 1$ and so, up to duality, we may assume that each block of lines has size $s^2 + 1$.
\end{proof}

Suppose $G$ is neither quasiprimitive on points nor on lines.  By Lemma \ref{lem:biptrans}, if every nontrivial normal subgroup of $G$ is transitive on either $\calP$ or $\calL$, then $G$ acts quasiprimitively on either $\calP$ or $\calL$.  Hence, we may assume that $G$ contains a nontrivial normal subgroup $N$ that is intransitive on both $\mathcal{P}$ and $\mathcal{L}$.  

\begin{lem} \label{lem:Nrestrictions}
 If $N$ is a nontrivial normal subgroup of $G$ that is intransitive on both $\mathcal{P}$ and $\mathcal{L}$, then $|N| = st + 1$.
\end{lem}

\begin{proof}
 The orbits of $N$ on $\mathcal{P}$ and $\mathcal{L}$ are both systems of imprimitivity.  By Lemma \ref{lem:bipintrans} (ii), the $N$-orbits on points and the $N$-orbits on lines each have have size $|N|$.  If $t > s$ or $s > t$, then by Lemma \ref{lem:Brestrictions} and duality $|N| = st + 1$.  On the other hand, if $s = t$, then by Lemma \ref{lem:blocks=t} $|N| = s^2 + 1$, as desired.
 \end{proof}

\begin{thm}
 \label{thm:quasi0poss}
 Let $\mathcal{Q}$ be a locally $(G,2)$-transitive generalized quadrangle.  Then $G$
 acts quasiprimitively on $\mathcal{P}$ or $\mathcal{L}$.
  \end{thm}

\begin{proof}
Let $\Gamma$ be the incidence graph of $\mathcal{Q}$.
Suppose $G$ does not act quasiprimitively on either $\calP$ or $\calL$.  By Lemma \ref{lem:biptrans}, $G$ has a nontrivial normal subgroup $N$ that is intransitive on both $\mathcal{P}$ and $\mathcal{L}$.

By Lemmas \ref{lem:bipintrans} and \ref{lem:Nrestrictions}, 
$\Gamma_N \cong K_{s+1, t+1}$, each $N$-orbit of points is an ovoid, and each $N$-orbit of lines is a spread. 
Let $P_1^N, \dots, P_{s+1}^N$ be the $N$-orbits of $\mathcal{P}$, which we will denote collectively as $\calP_N$, and $\ell_1^N, \dots, \ell_{t+1}^N$ be the $N$-orbits of $\mathcal{L}$, which we will denote collectively as $\calL_N$.  Let $\alpha$ be a point or a line, and let $B = \alpha^N$.  By the Frattini argument, $G_B = NG_\alpha$.

First, we note that $G/N$ is faithful on $\calP_N \cup \calL_N$.  Indeed, suppose $Ng$ fixes all blocks of points and all blocks of lines.  Then, $Ng \subseteq G_{P_1^N} = NG_{P_1}$, so there is $g_1 \in Ng$ such that $g_1$ fixes $P_1$.  Since all blocks of lines are fixed, all lines incident with $P_1$, i.e., all lines in $\Gamma(P_1)$, are fixed, and, by a connectivity argument, $g_1$ fixes all points and all lines.  Since $G$ is faithful on points and lines, $g_1 = 1$, $Ng = N$, and so $G/N$ is faithful on $\calP_N \cup \calL_N$.

We now divide into cases: either $G/N$ is unfaithful on both $\calP_N$ and $\calL_N$ (\textbf{Case 1}), $G/N$ is faithful on exactly one of $\calP_N$ or $\calL_N$ (\textbf{Case 2}), or $G/N$ is faithful on both $\calP_N$ and $\calL_N$ (\textbf{Case 3}). In any case, $G/N$ acts locally $2$-transitively on the complete bipartite graph $K_{s+1, t+1}$, which means we are in some case covered in \cite{localcompbip}. 

\textbf{Case 1.}  Suppose that $G/N$ is unfaithful on both $\calP_N$ and $\calL_N$.  By \cite[Theorem 5.3]{localcompbip}, we may assume that $G/N = (K_\calP  \times K_\calL).M$, where $K_\calP$ is the kernel of $G/N$ acting on $\calP_N$, $K_\calL$ is the kernel of $G/N$ acting on $\calL_N$, $K_\calP .M$ is $2$-transitive on $\calL_N$, and $K_\calL . M$ is $2$-transitive on $\calP_N$.  

First, if $K_\calP$ is $2$-transitive on $\calL_N$, then we claim that $G$ acts transitively on the antiflags of $\calQ$.  It suffices to show that, if $P = P_1$ is incident with $\ell = \ell_1$, then the stabilizer $G_{P\ell}$ of the flag $(P, \ell)$ is transitive on lines concurrent with $\ell$.  Since $G_P$ is transitive on collinear points, it further suffices to show that, if $P' \neq P$ is a second point incident with $\ell$, then $G_{P\ell P'}$ is transitive on lines incident with $P'$ other than $\ell$.  Indeed, if $NK_\calP$ is the preimage of $K_\calP$ in $G$, then $NK_\calP \cap G_P$ fixes $P$ and all blocks of points setwise but is still $2$-transitive on blocks of lines (since $K_\calP$ is $2$-transitive on $\calL_N$).  In particular, $NK_\calP \cap G_{P\ell P'}$ is still transitive on the blocks of lines other than $\ell^N$, and so $NK_\calP \cap G_{P\ell P'}$ is transitive on the lines incident with $P'$ other than $\ell$, and so $G$ acts transitively on the antiflags of $\calQ$.  However, the finite, thick antiflag-transitive generalized quadrangles were classified in \cite{BLS}, and none had such a system of blocks on points and lines, ruling out this case.  Moreover, unless $t+1 = 28$, by \cite[Theorem 1.1]{localcompbip}, this rules out either of $K_\calP$ or $K_\calL$ from being an almost simple group. 
 
Next, assume that each of $K_\calP .M$ and $K_\calL .M$ are affine, i.e., suppose that there exist subgroups $X_\calP \le K_\calP$ and $Y_\calL \le K_\calL$ that are each elementary abelian and regular on $\calL_N$ and $\calP_N$, respectively.  In particular, if $(P, \ell)$ is a flag, then this means that there exists $X \le G_P$ that is regular on $\Gamma(P)$, there exists $Y \le G_\ell$ that is regular on $\Gamma(\ell)$, $NX$ is regular on points, and $NY$ is regular on lines.  Moreover, since $[X,X]$ and $[Y,Y]$ each fix all blocks of points and lines, both $[X,X]$ and $[Y,Y]$ are subgroups of $N$.  Since $[X,X] \le N \cap G_P = 1$ and $[Y,Y] \le N \cap G_\ell = 1$, both $X$ and $Y$ are abelian groups.  

Let $1 \neq x \in X$, $1 \neq y \in Y$.  First, since $[y,x] = y^{-1}x^{-1}yx$ fixes each block of points and each block of lines, $[y,x] \in N$.  If $[y,x] = 1$, then $\ell^x = \ell^{yx} = \ell^{xy}$ and $P^y = P^{xy} = P^{yx}$.  Since $P^y \in \Gamma(\ell)$, $\ell^x \in \Gamma(P)$, and $P^y = P^{xy} \in \Gamma(\ell^{xy}) = \Gamma(\ell^x)$, $(P, \ell, P^y, \ell^x)$ is a $4$-cycle in $\Gamma$, meaning two distinct points lie on two distinct lines, a contradiction.  Hence no nonidentity element of $X$ commutes with any nonidentity element of $Y$.  

On the other hand, let $n$ be a nontrivial element of $N$.  Then, $(P, \ell^n)$ is an antiflag, and so, by the GQ Axiom, there is a unique path $(P, \ell', P', \ell^n)$ of length $3$ from $P$ to $\ell^n$ in $\Gamma$.  Now, $X$ is regular on $\Gamma(P)$, so $\ell' = \ell^x$ for some $x \in X$.  Moreover, if $P'$ is in the same $N$-orbit as $P^y$, where $y \in Y$, then $P' = P^{yx}$, since $\ell' = \ell^x = \ell^{yx}$.  Since $P' = P^{yx} = P^{x^{-1}yx}$, for similar reasons, we have that $\ell^n = \ell^{x^{-1}yx} = \ell^{y^{-1}x^{-1}yx} = \ell^{[y,x]}$, which implies that every nonidentity element of $N$ can be expressed as an element of $[Y,X]$.  On the other hand, $|N|-1 = st$, $|Y| - 1 = s$, and $|X| - 1 = t$, so in fact $N = \{[y,x]: y \in Y, x \in X\}$, and each nonidentity element of $N$ is uniquely expressible in the form $[y,x]$, where $y \in Y, x \in X$.

As we have seen,
\[ P^\perp = \{P^{yx} : y \in Y, x \in X \}.\]
Since $G_P$ is transitive on $P^\perp \backslash \{P\}$, $G_P$ fixes $P^N$ setwise, and every element of $P^N \backslash \{P\}$ is uniquely expressible as $P^{[y,x]}$, where $x \in X, y \in Y$, $G_P$ is transitive on $P^N \backslash \{P\}$.  This means $G_{P^N} = NG_P$ is $2$-transitive on $P^N$, and, since $G_{P^N}$ contains a regular normal subgroup $N$, we have $st + 1 = |N| = p^d$ for some prime $p$ and positive integer $d$. 

Now, since $[y,x]^{-1} = [x,y]$, we may also represent each point of $P^N\backslash \{P\}$ as $P^{[x,y]}$, and, using arguments similar to those above, $(\ell, P^y, \ell^{xy}, P^{[x,y]})$ is the unique path of length $3$ from $\ell$ to $P^{[x,y]}$ in $\Gamma$, and so there is a unique $4$-path in $\Gamma$ from $P$ to $P^{[x,y]}$ through $\ell$. Since $X$ fixes $\calP_N$ and is transitive on $\Gamma(P)$, for each line $\ell_i$ incident with $P$, there is a unique $4$-path starting with $P$ that goes through $\ell_i$ and ends at $P^{[x,y]}$; in particular, this means there is a unique point $Q_i$ on each $\ell_i$ that is collinear with $P^{[x,y]}$.  Moreover, since $X$ fixes $\calP_N$, all points collinear with both $P$ and $P^{[x,y]}$ lie in $(P^y)^N$, i.e.,
\[ \{P, P^{[x,y]} \}^\perp = \{P^{yx'}: x' \in X\}. \]
On the other hand, $x$ was arbitrary, so $|\{P, P^{[x,y]}\}^{\perp \perp}| = t+1$, and $(P, P^{[x,y]})$ is a \textit{regular pair} (see \cite[Section 1.3]{fgq}).  Now, the proof is exactly the same as the end of the proof of \cite[Theorem 4.1]{BLS}: by the dual argument, we also obtain a regular pair of nonconcurrent lines, and, by \cite[1.3.6(i)]{fgq}, $s = t$.  By \cite[1.8.5]{fgq}, since the generalized quadrangle $\calQ$ of order $s$ has a regular pair of noncollinear points and the point set $\calP$ can be partitioned into ovoids, $s$ must be odd.  Since $st + 1 = s^2 + 1 = |N| = p^d$ for some prime $p$, we have $p = 2$ and $s^2 + 1 = 2^d$.  On the other hand, if $d \ge 2$, then we have that $s^2 \equiv 3 \pmod 4$, a contradiction. If $d = 1$, then $s = 1$, a contradiction to $\calQ$ being thick.  Therefore, $K_\calP .M$ and $K_\calL .M$ cannot both be affine. 
 
Therefore, if $G/N$ is unfaithful on both $\calP_N$ and $\calL_N$, then without loss of generality, $s = 27$ and $K_\calL .M \cong \PGammaL(2,8)$ in its action on $28$ points and $K_\calL \cong \PSL(2,8)$.  Moreover, this implies that $6 \le t \le 27^2$ by Lemma \ref{lem:GQbasics}(iii), and, since $\calQ$ has an ovoid, namely $P^N$, $t \le s^2 - s = 702$ by \cite[1.8.3]{fgq}.  Also, by \cite[Theorem 1.1]{localcompbip}, $t+1 = 28$ or is a prime power.  By Lemma \ref{lem:GQbasics}(ii), $s+t$ divides $st(s+1)(t+1)$, which by direct calculation implies that $t \in \{12, 15, 27, 36, 162, 540  \}$, which in turn implies that
\[ |N| = st+1 = 27t + 1 \in \{325, 406, 730, 973, 4375, 14581 \}.\]
In each case, $N$ is solvable, i.e., there does not exist a nonabelian composition factor for these orders.  Moreover, in none of these cases does $\Aut(N)$ have a composition factor isomorphic to $\PSL(2,8)$, which implies that, if $Y := NK_\calL \cap G_\ell$, then $Y \le C_G(N)$.  Thus, for $n \in N$, $Y = Y^n \le (G_\ell)^n = G_{\ell^n}$, so $Y$ fixes every line in $\ell^N$.  Since the action of $\PSL(2,8)$ on $28$ elements is not regular, there exists a nontrivial $y \in Y$ fixing $P^N$.  However, since $y \in Y$, $y$ fixes every element of $\ell^N$, so $y$ fixes $P^N$ pointwise.  Since $y$ fixes $P^N$ pointwise and every block of lines setwise, $y$ fixes every line, a contradiction to $G$ acting faithfully. Therefore, $G/N$ must act faithfully on either $\calP_N$ or $\calL_N$.

\textbf{Case 2.}  Assume now that $G/N$ acts faithfully on exactly one of the biparts of $K_{s+1, t+1}$.  By \cite[Theorem 1.1]{localcompbip}, there are only a few possibilities for $s$ and $t$ up to duality.  If $s+1 = 3$ and $t+1 = q^3 + 1$, where $q$ is a prime power, then we have a contradiction to Lemma \ref{lem:GQbasics}(iii).  If $s+1 = r$ and $t+1 = (q^d - 1)/(q - 1)$ or $q^d$, where $q$ is a prime power, $d$ is a natural number, and $r$ divides $\gcd(d,q-1)$, then
\[ 2^{d-1} - 1 \le q^{d-1} - 1 \le \frac{q^d - q}{q-1} \le t \le s^2 = (d - 1)^2,\]
and so $d \le 5$.  However, this implies \[s = r-1 \le \gcd(d,q-1) - 1 \le d - 1 \le 4,\] and these possibilities are ruled out by inspection.  If $s+1 = r$ and $t+1 = p^f$, where $r-1$ divides $f$, then $s \le f$, then
\[2^f - 1 \le p^f - 1 = t \le s^2 \le f^2, \]
and so $s \le f \le 4$, and these possibilities are ruled out by inspection.  Finally, it is possible that $s+1$ and $t+1$ are one of the exceptional cases listed in \cite[Table 5]{localcompbip}.  In these instances, the only possibilities for $(s,t)$ that satisfy Lemma \ref{lem:GQbasics}(ii) and (iii), i.e., $s \le t \le s^2$ and $(s+t) \mid st(s+1)(t+1)$, are $(s,t) \in \{(7,7), (5, 15), (6, 15)\}.$  In these cases, we have
\[|N| = st + 1 \in \{50, 76, 91\}. \]  The full information in each case is summarized in the following table:

\begin{center}
 \begin{tabular}{lccc}
 \toprule
 $G/N$ & $s$ & $t$ & $|N|=st+1$ \\
 \midrule
 $2^3:\GL(3,2)$& $7$  & $7$ & $50$ \\
 $2^4:A_6$ & $5$ & $15$ & $76$ \\
 $2^4:A_7$ &  $6$ & $15$ & $91$\\
 \bottomrule
 \end{tabular}
\end{center}

Notice that in each case, $N$ is solvable (i.e., there cannot be a nonabelian composition factor for these orders),
and indeed, $\Aut(N)$ is solvable, too, in each of these cases. Therefore, $G/C_G(N)$ is solvable (as it
embeds naturally into $\Aut(N)$), and therefore it follows that 
$G$ is a central extension of $G/N$ by $N$. However, $N$ cannot lie in the center of 
 $G$, since otherwise, $G$ would fix each of the $N$-orbits, a contradiction.  

 \textbf{Case 3.}  Finally, we assume that $G/N$ acts faithfully on each of $\calP_N$ and $\calL_N$.  Since $G/N$ acts faithfully and locally $2$-arc-transitively on $K_{s+1, t+1}$, we have from  \cite[Theorem 1.1]{localcompbip}
 the following possibilities (see also \cite[Table 4]{localcompbip}):
\begin{center}
 \begin{tabular}{lccc}
 \toprule
 $G/N$ & $s$ & $t$ & $|N|=st+1$ \\
 \midrule
 $A_6$& $5$  & $5$ & $26$ \\
 $\AGL(3,2)$ & $7$ & $7$ & $50$ \\
 $M_{12}$ &  $11$ & $11$ & $122$\\
 $\PSL(4,2)$ & $7$ & $14$ & $99$\\
 \bottomrule
 \end{tabular}
\end{center}
These cases are ruled out like the cases above: in each case, $N$ is solvable,
and $\Aut(N)$ is solvable, too. Therefore, $G/C_G(N)$ is solvable (as it
embeds naturally into $\Aut(N)$), and therefore it follows that 
$G$ is a central extension of $G/N$ by $N$. However, $N$ cannot lie in the center of 
 $G$, since otherwise, $G$ would fix each of the $N$-orbits, a final contradiction, showing that $G$ must act quasiprimitively on either $\calP$ or $\calL$. 
\end{proof}

%
%
%
%


\section{Quasiprimitive on points but not lines}
\label{sect:single}

In this section, we classify the locally $2$-transitive generalized quadrangles with a collineation group that is quasiprimitive on points but not on lines.

\begin{hyp}
\label{hyp:GQ1quasi}
Let $\mathcal{Q}$ be a finite thick generalized quadrangle of order $(s,t)$.  Let $\Gamma$ be the associated incidence graph, and suppose that there exists a subgroup $G \le \Aut(\Gamma)$ such that $\Gamma$ is locally $(G,2)$-arc-transitive.  We will denote by $\mathcal{P}$ the points of $\mathcal{Q}$ and by $\mathcal{L}$ the lines of $\mathcal{Q}.$  Throughout, we will abuse notation a bit and use $\mathcal{P}$ and $\mathcal{L}$ to refer to the the biparts of $\Gamma$ as well.  Finally, assume that $G$ acts quasiprimitively on $\mathcal{P}$ but not on $\mathcal{L}$.
\end{hyp}

\begin{lem}
\label{lem:t+1orbs}
 Assume Hypothesis \ref{hyp:GQ1quasi}.  There exists a nontrivial normal subgroup $N$ of $G$ such that $N$ is transitive on $\calP$ and intransitive on $\calL$, and, moreover, any such normal subgroup $N$ has exactly $t+1$ orbits on $\calL$. 
\end{lem}

\begin{proof}
 Since $G$ acts quasiprimitively on $\calP$ but not on $\calL$, there exists a normal subgroup $N$ that is transitive on $\calP$ but not on $\calL$.  
 
 Let $P$ be a point in $\calP$.  By Lemma \ref{lem:loc2trans}, the stabilizer $G_P$ of $P$ is $2$-transitive on the lines incident with $P$, so $G_P$ is primitive on incident lines; in particular, the $N$-orbits of lines are a $G$-invariant system of imprimitivity on $\calL$, so this implies that either all lines incident with $P$ are in the same $N$-orbit of lines, or all lines incident with $P$ are in distinct $N$-orbits.  If all lines incident with $P$ are in the same $N$-orbit, say $\ell^N$, then, by the transitivity of $N$ on $\calP$, every line in $\calL$ is is $\ell^N$, i.e., $\ell^N = \calL$, which is a contradiction to the intransitivity of $N$ on $\calL$.  Hence, the lines incident with $P$ are in $t+1$ distinct $N$-orbits, and, by the transitivity of $N$ on $\calP$, $N$ has exactly $t+1$ orbits on $\calL$.   
\end{proof}

\begin{lem}
 \label{lem:HAon1}
 Assume Hypothesis \ref{hyp:GQ1quasi}.  If $G$ acts with type $\HA$ on $\mathcal{P}$, then $\mathcal{Q}$ is the unique generalized quadrangle of order $(3,5)$.
\end{lem}

\begin{proof}
 Since $\Gamma$ is locally $(G,2)$-arc-transitive, $G$ is transitive on lines.  Every group of type $\HA$ is primitive (see \cite[Section 5]{FinQuasiprimGraphs}), and the finite thick generalized quadrangles admitting an automorphism group that is point-primitive with type $\HA$ and line-transitive were classified in \cite[Corollary 1.5]{GQtranshyp}: namely, they are the unique generalized quadrangle of order $(3,5)$ or the generalized quadrangle of order $(15,17)$ arising from the Lunelli-Sce hyperoval.  The incidence graph $\Gamma$ arising from the unique generalized quadrangle of order $(3,5)$ is in fact locally 3-arc-transitive \cite[Theorem 1.1]{BLS}.  On the other hand, since there are $16$ points incident with each line in the generalized quadrangle of order $(15,17)$ arising from the Lunelli-Sce hyperoval, a group acting locally 2-arc-transitively on the incidence graph is necessarily divisible by $16 \cdot 15$, whereas the order of the colineation group of this generalized quadrangle is not divisible by $5$ \cite{BichMazzSomma, BrownCherowitzo}.  The result follows. 
\end{proof}

\begin{lem}
 \label{lem:notHSon1}
 Assume Hypothesis \ref{hyp:GQ1quasi}.  Then, $G$ cannot act on $\mathcal{P}$ with type $\HS$.
\end{lem}

\begin{proof}
 A group $G$ that is quasiprimitive with type $\HS$ is also primitive, and so the result follows from \cite[Theorem 1.1]{BambergPopielPraeger}.
\end{proof}

\begin{lem}
 \label{lem:notASon1}
 Assume Hypothesis \ref{hyp:GQ1quasi}.  Then, $G$ cannot act on $\mathcal{P}$ with type $\AS$.
\end{lem}

\begin{proof}
 Since $G$ is almost simple, there is a finite simple group $T$ such that $T \le G \le \Aut(T)$.  Since $T$ is transitive on $\calP$ but not on $\calL$, by Lemma \ref{lem:t+1orbs}, $T$ must have exactly $t+1$ orbits on $\calL$.
 By \cite[Theorem 1.3]{localstar}, the possibilities for $\{T, t+1\}$ are: $\{\PSL_n(q), t+1\}$, where $n \ge 3$ and $t+1$ is an odd prime dividing $\gcd(n,q-1)$; $\{\PSU_n(q), t+1\}$, where $n \ge 3$ and $t+1$ is an odd prime dividing $\gcd(n, q+1)$; $\{\POmega^+_8(q), t+1\}$, where $t+1$ is either $3$ or $4$; $\{E_6(q), 3\}$; or $\{ {}^2E_6(q), 3\}$.  Since all generalized quadrangles with $t \in \{2,3\}$ are known \cite[\S6.1,\S6.2]{fgq}, and none of them admits $ \POmega^+_8(q)$, $E_6(q)$, or ${}^2E_6(q)$ for any $q$, these cases are immediately ruled out.
 
 Suppose that $T = \PSL_n(q)$.  We also observe that $t+1$ divides $\gcd(n,q-1)$.  If $n \ge 6$, then the smallest permutation representation of $T$ is on at least $(q^6 - 1)/(q - 1) = q^5 + q^4 + q^3 + q^2 + q + 1$ points, so 
 \[(s+1)(st+1) = |\mathcal{P}| \ge q^5 + q^4 +q^3 + q^2 + q + 1.\]
 This implies that $s \ge q$, and, since $t < q-1$, we have
 $s^2q > q^5$, which implies $s > q^2$.  However, this means $\sqrt{s} > q > t$, a contradiction to Lemma \ref{lem:GQbasics} (iii).  This means $n \in \{3,4,5\}$.  Since $t+1$ is an odd prime that divides $\gcd(n,q-1)$, we have $t=2$ or $t = 4$.  We saw above that $t = 2$ is impossible in this case, and, if $t = 4$, then $s \le 16$, which means $|\mathcal{P}| \le (16 + 1)(64 + 1) = 1105$.  On the other hand, $|\mathcal{P}| \ge q^4 + q^3 + q^2 + q + 1 > q^4$, and so $q^4 < 1105$, i.e., $q \le 5$.  However, $t+1 = 5$ and must divide $q - 1 \le 4$, a contradiction. 
 
 Finally, suppose $T = \PSU_n(q)$.  The proof proceeds similarly: we observe that $t+1$ divides $\gcd(n,q+1)$.  Note that $q = 2$ implies $t = 2$, and there are no examples in this case.  If $n \ge 5$, then the smallest permutation representation of $T$ is on at least $(q^5+1)(q^2+1)$ points, so
 \[(s+1)(st+1) = |\mathcal{P}| \ge (q^5 + 1)(q^2 + 1),\] which as above implies $\sqrt{s} > t$, a contradiction.  Hence $n = 3$ or $n = 4$, and, since $k$ is an odd prime dividing $n$, we conclude that $t+1 = n = 3$.  However, this means $t = 2$, a final contradiction.  Therefore, there are no such generalized quadrangles.
\end{proof}

Before proceeding, we prove the following lemma, which is useful in the remaining cases when $\soc(G) \cong T^k$, where $T$ is a nonabelian finite simple group. 

\begin{lem}
\label{lem:lognk}
 Let $X \le \Sym(k)$ and suppose $X$ acts $2$-transitively on $\Omega$, where $|\Omega| = n > k$.  Then, $\log_2 n < k$.  In particular, if $X^\Omega$ is affine, then $\log_2 n \le k/2$, and, if $X^\Omega$ is almost simple, then $n < 2k$ unless $n = 28$, $k = 9$, and $X^\Omega \cong \PGammaL(2,8)$. 
\end{lem}

\begin{proof}
 Let $|\Omega| = n$, and assume $X^\Omega$ is $2$-transitive, where $X \le \Sym(k)$ for some $k < n$.  If the action of $X^\Omega$ is affine, then $n = p^f$, where $p$ is a prime.  Thus there are $f$ generating elements in $X$ of order $p$ that commute, and so 
 \[k \ge pf \ge f \cdot (2\log_2 p) = 2\log_2 n.\]  If $X$ is not affine, then $X^\Omega$ is one of a finite list of groups; see \cite[Table 7.4]{CameronPerm}.  Unless $n = 28$ and $k = 9$, $n < 2k$ by \cite[Table 5.2A]{KleidmanLiebeck}, and so $\log_2 n < n/2 < k$.  Finally, when $n = 28$, $\log_2 28 < 5 < 9$, and the result holds in any case.
\end{proof}

\begin{lem}
 \label{lem:notTWon1}
 Assume Hypothesis \ref{hyp:GQ1quasi}.  Then, $G$ cannot act on $\mathcal{P}$ with type $\TW$.
\end{lem}

\begin{proof}
Assume by way of contradiction that $G$ is quasiprimitive with type $\TW$ on $\mathcal{P}$.
Then $G$ has a unique minimal normal subgroup $N = T_1 \times T_2 \times \cdots \times T_k \cong T^k$, where $k \ge 2$ and $T$ is a nonabelian finite simple group. Since $N$ acts regularly on $\mathcal{P}$, we have $|\mathcal{P}| = |N|$ and $G/N \liso \Sym(k)$.  On the other hand, if $P \in \mathcal{P}$, $G_P$ is $2$-transitive on the $t+1$ lines incident with $P$, and since $G = NG_P$, it must be that $\Sym(k)$ has a subgroup with a $2$-transitive action on $t+1$ elements. 
 By Lemma \ref{lem:lognk}, $t+1 < 2^k$ and hence
 \[ |T|^k = |\mathcal{P}| = (s+1)(st+1) < (t+1)^5 < 32^k,\]
 a contradiction, since $T$ is a nonabelian simple group.  
\end{proof}




\begin{lem}
 \label{lem:notPAon1}
  Assume Hypothesis \ref{hyp:GQ1quasi}.  Then, $G$ cannot act on $\mathcal{P}$ with type $\PA$.
\end{lem}

\begin{proof}
Assume Hypothesis \ref{hyp:GQ1quasi} and furthermore suppose that $G$ acts quasiprimitively with type PA on points. The group $G$ has a unique minimal normal subgroup $N = T_1 \times T_2 \times \dots \times T_k \cong T^k$, where $k>1$ and $T$ is a nonabelian simple group; $G \leqslant (H_1 \times H_2 \times \dots \times H_k) \rtimes \Sym(k) \cong H \Wr \Sym(k)$, where $T \leqslant H \leqslant \Aut(T)$; and $G$ acts transitively by conjugation on the simple direct factors of $N$.  The group $G$ preserves a product structure $\Delta^k$ on $\mathcal{P}$, and $N_P$ is a subdirect subgroup of the stabilizer $N_B \cong T_\delta^k$, where $B = (\delta, \delta, \dots, \delta)$ for $\delta \in \Delta$, i.e., $N_P$ projects onto $T_\delta$ in each coordinate.  
Note that the minimal normal subgroup $N$ of $G$ has exactly $t+1$ orbits on $\mathcal{L}$ since $N$ is transitive on $\calP$ but not $\calL$ by Lemma \ref{lem:t+1orbs}.

 Since the quotient graph $\Gamma_N$ is the complete bipartite graph $K_{1,t+1}$ and $\Gamma$ is locally $(G,2)$-arc-transitive, if $P \in \mathcal{P}$, then $N_P^{\Gamma(P)} = 1$, i.e., $N_P \le \stabkernel{G}{P}$.  Define $Y:= G \cap (H_1 \times H_2 \times \dots \times H_k)$, i.e., $Y$ is the (normal) subgroup of $G$ that fixes each of the $k$ coordinates.  

 Suppose $Y_P^{\Gamma(P)} \neq 1$.  Since $Y_P^{\Gamma(P)} \lhd G_P^{\Gamma(P)}$, we have $\soc(G_P^{\Gamma(P)}) \le Y_P^{\Gamma(P)}$, which implies $Y_P$ is transitive on $\Gamma(P)$.  By Burnside's Theorem, $\soc(G_P^{\Gamma(P)})$ is either elementary abelian and regular or a nonabelian finite simple group.  Since $N_P \le \stabkernel{G}{P}$, $Y_P^{\Gamma(P)} \cong Y_P/Y_{(\Gamma(P))}$, and $Y_P/N_P$ is solvable by the Schreier Conjecture, $G_P^{\Gamma(P)}$ is affine and $t+1 = q^k$, where $q$ is a prime power dividing $|H/T|$.  If $G_P^{\Gamma(P)}$ is not solvable, then a nonabelian finite simple group is involved in $\Sym(k)$, and so $k \ge 5$.  Examining \cite[Table 7.3]{CameronPerm} and noting that none of $G_2(q)$, $\PSU(3,3)$, or $\PSL(2,13)$ are involved in $\Alt(6)$, we see that one of $\PSL(k,q)$ or $\PSp(k,q)$ must be involved in $\Sym(k)$.  However, by \cite[Table 5.2A]{KleidmanLiebeck}, this is impossible when $k \ge 5$.  If $G_P^{\Gamma(P)}$ is solvable, then, since none of the groups arising from near-fields have the structure $(H/T) \Wr \Sym(2)$, we may assume that $G_P^{\Gamma(P)} \liso \AGammaL(1,q^k)$.  However, by, for instance, \cite[Proposition 2.7]{MSV}, this means $\Sym(k)$ contains a cyclic group of order at least $(q^k - 1)/k$.  This implies that $(q^k - 1)/k \le k$, which only holds when $q = 2$ and $k < 5$. However, $k \nmid (q^k - 1)$ in this instance, which implies $2^k - 1 < k$, a contradiction.  Hence $Y_P^{\Gamma(P)} = 1$.
 
  Since $Y_P^{\Gamma(P)} = 1$, we have $Y_P \le \stabkernel{G}{P}$.  Moreover, since $G_P/Y_P \liso G/Y \liso \Sym(k)$ and $G_P^{\Gamma(P)} \cong G_P/\stabkernel{G}{P}$, a subgroup of $\Sym(k)$ has a $2$-transitive action on $t+1$ points.  If $t+1 \le k$ and $|\Delta| = d$, then by Lemma \ref{lem:Pbounds}(iii),
 \[d^k = |\Delta|^k \le |\mathcal{P}| < (t+1)^5 < k^5.\]
 This implies $d < k^{\frac{5}{k}}$.  Since $d = |T:T_\delta|$, a finite simple group must have a permutation representation on at least $d$ points, and so $k^{\frac{5}{k}} > d \ge 5$, which implies $k < 5$.  However, this means $t \le 3$, a contradiction.
 Therefore, $t+1 > k$.  

 Consider first the case when $G_P^{\Gamma(P)}$ is affine.  By Lemma \ref{lem:lognk}, $\log_2(t+1) \le k/2$.  On the other hand, since $|T:T_\delta| < (t+1)^\frac{5}{k}$, 
 \[\log_2 |T:T_\delta| < \frac{5}{k} \log_2 (t+1) \le \frac{5}{2},\]
 so $|T:T_\delta| < 6$.  This means $T$ is isomorphic to a subgroup of $\Sym(5)$, i.e., $T \cong \Alt(5)$ and $|\Delta| = 5$.  Since $t < 2^\frac{k}{2} - 1$, $s \le t^2 < (2^\frac{k}{2} - 1)^2$; in particular, $(s+1) < 2^k$.  Also, $N_P$ is a subdirect subgroup of $A_4^k$, and $5^k \mid |\mathcal{P}|$.  Since $(st+1) \le 2^\frac{3k}{2} < 5^k$, this implies that $5 \mid (s+1)$ and $5 \mid (st+1)$.  If $G_\ell^{\Gamma(\ell)}$ is almost simple, then, since $N_\ell$ is transitive on $\Gamma(\ell)$, we have $\soc(G_\ell^{\Gamma(\ell)}) \cong A_5$, and so $s+1 \in \{5,6\}$.  If $s+ 1 = 6$, then $5$ is coprime to $|\mathcal{P}| = 6(5s+1)$, a contradiction.  So $s +1 = 5$, which means 
 \[5^k = |\Delta|^k \le |\mathcal{P}| = 5(4t+1) \le 5 \cdot 65 = 325.\]  Hence $k = 2,3$, a contradiction, since a $2$-transitive group on more than $k$ points would be involved in $\Sym(k)$ for $k = 2$ or $3$.  Thus $G_\ell^{\Gamma(\ell)}$ is also affine, and so $s+1 = 5^c$ for some $c < k$.  On the other hand, this means that $G_{P,\ell}$ is involved in $\Sym(k)$.  Since $5^k \le (s+1)(st+1) < 5^{2c}t$, we have
 \[ 5^{k-2c} \le t+1 < 2^\frac{k}{2},\] which implies that $c > 0.39k$.  If $G_\ell^{\Gamma(\ell)}$ is affine but not solvable, then by \cite[Table 7.3]{CameronPerm} either $\SL(m, 5^\frac{c}{m})$ or $\Sp(m, 5^\frac{c}{m})$ is involved in $\Sym(k)$, where $m \mid c$.  In either case, any
 permutation representation of such a group has degree at least $(5^c - 1)/(5^\frac{c}{m} - 1) > 5^\frac{c}{2}$, which implies
 \[5^{0.195k} < 5^\frac{c}{2} < k.\]  However, this implies $k \le 5$ and so $k = 5$.  However, neither $\SL(m, 5^\frac{c}{m})$ or $\Sp(m, 5^\frac{c}{m})$ is involved in $\Sym(5)$, and so we may assume $G_\ell^{\Gamma(\ell)}$ is solvable.  If $G_\ell^{\Gamma(\ell)}$ is solvable, then either $s+1 = 25$ (in which case $k = 2$, a contradiction) or $(5^c - 1)/c \le k$.    However, for all $k \ge 2$,
 \[ k - \frac{5^{0.39k} - 1}{0.39k} < 0,\] a contradiction.  Hence $G_P^{\Gamma(P)}$ is not affine.
 
 We next assume that $G_P^{\Gamma(P)}$ is almost simple.  By Lemma \ref{lem:lognk}, unless $t+1 = 28$ and $k = 9$, $t+1 < 2k$, which implies by Lemma \ref{lem:Pbounds}(iii) that
 \[5^k \le d^k = |\Delta|^k \le |\mathcal{P}| < (t+1)^5 < (2k)^5,\] which implies $k \le 8$.  This gives $k = 5$ and $t+1 = 6$, $k = 6$ and $t+1 = 10$, $k = 7$ and $t+1 = 8$, or $k = 8$ and $t+1 = 15$.  Consider first the case when $k = 5$ and $t+1 = 6$.  This means $s \le 25$, and $d^5$ divides $|\mathcal{P}| = (s+1) (5s + 1) \le 3276$; however, $d^5 \ge 5^5 = 3125$, so $d^5 = |\mathcal{P}| = 3125$; a contradiction since $5s+1$ is coprime to $5$. Next, consider the case when $k = 6$ and $t+1 = 10$.   This means $s \le 81$, and $d^6$ divides $|\mathcal{P}| = (s+1) (9s + 1) \le 59860$; hence $d = 5$ or $d = 6$.  If $d = 5$, then $15625 \mid |\mathcal{P}|$, and so $|\mathcal{P}| = 5^6, 2\cdot 5^6,$ or $3 \cdot 5^6$.  However, $(s+1)(9s+1) = c\cdot 5^6$ does not have an integer root when $c = 1, 2, 3,$ a contradiction. 
 We next consider the case when $k = 7$ and $t+1 = 8$.  This means $s \le 49$, but \[78125 = 5^7 \le d^7 \le |\mathcal{P}| = (s+1) (7s + 1) \le 17200,\] a contradiction.  Suppose now that $k = 8$ and $t+1 = 15$.  This means $s \le 196$, and $d^8$ divides $|\mathcal{P}| = (s+1)(14s+1) \le 540765,$ i.e., $d = 5$.  Since $5^8 = 390625$, this means $(s+1)(14s+1) = 5^8$, a contradiction, since there are no integer solutions.
 Finally, assume that $t+1 = 28$ and $k = 9$.  This means $s \le 729$, and $d^9$ divides $|\mathcal{P}| = (s+1)(27s + 1) \le 14369320 < 7^9$.  Hence $d = 5$ and $|\mathcal{P}| = c \cdot 5^9$, where $1 \le c \le 7$, or $d = 6$ and $|\mathcal{P}| = 6^9$.  In each of these situations, $(s+1)(27s+1) = |\mathcal{P}|$ has no integer solutions.  Therefore, no such generalized quadrangle is possible.
\end{proof}

We can now prove the main result of this section, which is a classification of the locally $(G,2)$-transitive generalized quadrangles where $G$ is only quasiprimitive on points.

\begin{thm}
 \label{thm:quasi1}
 Let $\mathcal{Q}$ be a thick locally $(G,2)$-transitive generalized quadrangle such that $G$ is quasiprimitive on points but not on lines.  
 Then, $\mathcal{Q}$ is the generalized quadrangle of order $(3,5)$.
\end{thm}

\begin{proof}
 By Lemma \ref{lem:types1}, if the incidence graph $\Gamma$ is locally $(G,2)$-arc-transitive, but $G$ is only quasiprimitive on $\mathcal{P}$, then $G$ acts on $\mathcal{P}$ with type $\HA$, $\HS$, $\AS$, $\TW$, or $\PA$.  The result follows from Lemmas \ref{lem:HAon1}, \ref{lem:notHSon1}, \ref{lem:notASon1}, \ref{lem:notTWon1}, and \ref{lem:notPAon1}.
\end{proof}

\section{Quasiprimitive on both points and lines}
\label{sect:both}

This section is dedicated to the characterization of the locally $2$-transitive generalized quadrangles with a collineation group that is quasiprimitive on both points and lines.

\begin{thm}\label{thm:QPonboth1}
 Let $\mathcal{Q}$ be a thick locally $(G,2)$-transitive generalized quadrangle such that $G$ is quasiprimitive on both points and lines.  
Then $G$ cannot act with type $\{\HA, \HA \}$, $\{\TW, \TW\}$, or $\{\SD, \PA\}$.
\end{thm}

\begin{proof}\ \\
\textbf{Case $\{\HA, \HA\}$:}
The quasiprimitive groups of type HA are also primitive (see \cite[Section 5]{FinQuasiprimGraphs}), and a group $G$ cannot act primitively with type $\{\HA, \HA\}$ on the points and lines of a finite generalized quadrangle \cite[Lemma 3.5]{primGQ}. 

\noindent \textbf{Case $\{\TW, \TW\}$:}
The proof is exactly the same as \cite[Lemma 3.5]{primGQ}, where it was shown that a group $G$ cannot act primitively with type $\{\TW, \TW\}$ on the points and lines of a finite generalized quadrangle, except the word ``primitive'' is replaced by the word ``quasiprimitive.''

\noindent \textbf{Case $\{\SD, \PA\}$:}
We obtain the possible vertex valencies from {\cite[Theorems 1.1, 1.2]{localdifferent}}, which are
$\{q^d, q+1\}, \{q, (q^n - 1)/(q-1)\}, \{q^d, q^2 + 1\}, \{q^d, q^3 + 1\},$ where $d,n$ are positive integers, $n \geqslant 3$, and $q$ is a prime power. 

Suppose first that the vertex valencies are $\{q^d, q+1\}$.  By Lemma \ref{lem:GQbasics}(iii), $q^d \leqslant (q+1)^2 = q^2 + 2q + 1$, and so $d = 1$ or $d = 2.$  If $d = 1$, then without a loss of generality, $s+1 = q+1$ and $t+1 = q$, and so $s = q$ and $t = q-1.$  By Lemma \ref{lem:GQbasics}(ii), this implies that $2q-1$ divides $q(q-1)(q+1)q = q^2(q^2-1)$.  On the other hand,
$$16q^2(q^2 - 1) = (2q-1)(8q^3 +4q^2-6q -3) - 3,$$ which implies that $2q-1$ is a divisor of $3$.  This implies that $q=1$ or $q=2$, a contradiction to $s,t > 1.$  If $d = 2$, then without a loss of generality, $s+1 = q+1$ and $t+1 = q^2$, and so $s = q$ and $t= q^2 - 1$.  By Lemma \ref{lem:GQbasics}(ii), this implies that $q^2+q-1$ divides $q^6+q^5 - q^4 - q^3.$  On the other hand,
$$q^6 +q^5 - q^4 - q^3 = (q^2+q-1)(q^4-q+1) + (1 - 2q),$$ which implies that $q^2+q-1$ is a divisor of $2q-1$, i.e., that $q^2+q-1 \leqslant 2q-1$ since $2q-1$ has no integral roots.  This means that $q^2 - q \leqslant 0$, a contradiction if $s,t > 1.$

Suppose now that the vertex valencies are $\{q, (q^n-1)/(q-1)\}$.  Without a loss of generality, $s = q-1$ and $t = (q^n - 1)/(q-1) - 1 = q^{n-1} + q^{n-2} + \cdots + q$.  Since $n \geqslant 3$, by Lemma \ref{lem:GQbasics}(iii) we have:
$$ q^2 + q \leqslant t \leqslant s^2 = q^2 - 2q + 1,$$ which is a contradiction.

Suppose now that the vertex valencies are $\{q^d,q^2+1\}$.  Without a loss of generality, $s = q^2$, and since $\sqrt{s} \leqslant t \leqslant s^2$ by Lemma \ref{lem:GQbasics}(iii), this means that $d = 2,3,$ or $4$.  If $d=2$, then $t = q^2 - 1$. Substituting in $r=q^2$, this reduces to the case $s = r$ and $t=r-1$, which was ruled out above.  If $d=3$, then $t=q^3 -1$, and by Lemma \ref{lem:GQbasics}(ii) this implies that $q^3 + q^2 -1$ divides $q^5(q^2 + 1)(q^3 - 1)$.  On the other hand,
$$q^5(q^2 + 1)(q^3 - 1) = (q^3 + q^2 -1)(q^7 - q^6 +2q^5 -2q^4 +q^3 -2q +3) + (-3q^2 -2q+3)$$ which implies that $q^3 + q^2 - 1$ divides $3q^2 + 2q - 3$.  Since $3q^2 + 2q - 3$ has no integral roots, this means that $q^3 + q^2 - 1 \leqslant 3q^2 + 2q - 3$, which implies that $q \leqslant 2$ if $q$ is an integer.  If $q=2$, however, $q^3+q^2-1$ does not divide $3q^2 + 2q - 3$, a contradiction.  Finally, if $d=4$, then, proceeding as above, we would have that $s= q^2$ and $t = q^4-1$, but setting $r= q^2$, this means that $s = r$ and $t = r^2-1$, which was ruled out above.

Finally, we suppose that the vertex valencies are $\{q^d, q^3 + 1\}$.  Without a loss of generality, assume that $s = q^3$.  Again, using Lemma \ref{lem:GQbasics}(iii), this implies that $d = 2, 3, 4, 5,$ or $6$.  When $d=2$, we have $t = q^2-1$, and so Lemma \ref{lem:GQbasics}(ii) implies that $q^3+q^2-1$ divides $q^5(q^2-1)(q^3+1)$.  On the other hand,
$$q^5(q^2-1)(q^3+1) = (q^3 + q^2 -1)(q^7-q^6+2q^4-3q^3+2q^2-3) + (5q^2-3),$$ which implies that $q^3 + q^2 -1$ divides $5q^2-3$.  Since $5q^2-3$ has no integral roots, this means $q^3 + q^2 - 1 \leqslant 5q^2 - 3$, i.e., that $q < 4.$  On the other hand, for no prime power $q \leqslant 3$ is $(5q^2 - 3)/(q^3 + q^2 - 1)$ an integer, which is a contradiction.  If $d=3$, then setting $r= q^3$ yields $s=r$ and $t=r-1$, which was ruled out above.  If $d = 4$, then $t = q^4-1$, and Lemma \ref{lem:GQbasics}(ii) implies that $q^4+q^3-1$ divides $q^7(q^4-1)(q^3+1)$.  On the other hand,
$$q^7(q^4-1)(q^3+1) = (q^4+q^3-1)(q^{10} - q^8 + q^8 -q^5 + 2q^4 - 3q^3 + 3q^2 - 4q +6) + (-9q^3 + 3q^2 - 4q + 6),$$ which implies that $q^4+q^3-1$ divides $9q^3 - 3q^2 + 4q - 6$.  Since $9q^3 - 3q^2 + 4q - 6$ has no integral roots, this means that $q^4+q^3 - 1 \leqslant 9q^3 - 3q^2 + 4q - 6$, which only holds for $q \leqslant 7$.  However, for no prime power $q \leqslant 7$ is $(9q^3 - 3q^2 + 4q -6)/(q^4+q^3-1)$ an integer, which is a contradiction.  If $d=5$, then $t=q^5-1$, and Lemma \ref{lem:GQbasics}(ii) implies that $q^5+q^3-1$ divides $q^8(q^5-1)(q^3+1)$.  On the other hand,
$$q^8(q^5-1)(q^3+1) = (q^5 + q^3 -1)(q^{11} - q^9 + q^8 + q^7 - q^6 - q^5 + q^3 + q^2 - 2q - 2) + (2q^4 + 3q^3 + q^2 - 2q - 2),$$ which implies that $q^5 + q^3 -1$ divides $2q^4 + 3q^3 + q^2 - 2q - 2 = (q+1)(2q^3 + q^2 - 2)$.  Since $(q+1)(2q^3 + q^2 - 2)$ has no positive integral roots, this implies that $q^5 + q^3 - 1 \leqslant 2q^4 + 3q^3 + q^2 - 2q - 2,$ which in turn means that $q = 2.$  Since $(2q^4 + 3q^3 + q^2 - 2q - 2)/(q^5 + q^3 -1)$ is not an integer when $q=2$, we have a contradiction.  Finally, if $d=6$, we have $t=q^6-1$.  However, setting $r = q^3$, this means that $s=r$ and $t=r^2-1$, which was ruled out above.  Since we have covered all possible vertex valencies, there cannot be such a finite generalized quadrangle.
\end{proof}

Because we will be repeatedly be making the same assumptions, for essentially the rest of this section we will be assuming the following hypothesis.

\begin{hyp}
\label{hyp:GQPA}
 Let $\mathcal{Q}$ be a thick locally $(G,2)$-transitive generalized quadrangle such that $G$ is quasiprimitive on both points and lines,
 with both actions of type PA.
 The group $G$ has a unique minimal normal subgroup $N = T_1 \times T_2 \times \dots \times T_k \cong T^k$, where $k>1$ and $T$ is a nonabelian simple group; $G \leqslant H \Wr \Sym(k)$, where $T \leqslant H \leqslant \Aut(T)$; and $G$ acts transitively by conjugation on the simple direct factors of $N$.  The group $G$ preserves a product structure $\Delta^k$ on $\mathcal{P}$ and a product structure $\Sigma^k$ on $\mathcal{L}$, and $N_P$ is a subdirect subgroup of the stabilizer $N_B \cong T_\delta^k$, where $B = (\delta, \delta, \dots, \delta)$ for $\delta \in \Delta$, i.e., $N_P$ projects onto $T_\delta$ in each coordinate. In this case, each $k$-tuple of elements from $\Delta$ is a block of points of $\mathcal{P}$.  (Similarly, $N_\ell$ projects onto $T_\epsilon$ in each coordinate for some $\epsilon \in \Sigma$, where $\ell$ is in the block $B'$ with stabilizer isomorphic to $T_\epsilon^k$.)  Finally, assume $s \leqslant t$.
\end{hyp}

We say that $N_P$ is \textit{diagonal} if $N_P \cong T_\delta$.

\begin{lem}
\label{lem:Hamming}
Assume Hypothesis \ref{hyp:GQPA}.  Let $P_1$ and $P_2$ (respectively, $\ell_1$ and $\ell_2$) be two collinear points (respectively concurrent lines) contained in blocks $B_1$ and $B_2$, and suppose that the Hamming distance between blocks $B_1$ and $B_2$ is $d$.  Then any two collinear points $P$ and $P'$ (respectively, concurrent lines $\ell$ and $\ell'$) are contained in blocks that are Hamming distance $d$ apart. 
\end{lem}

\begin{proof}
We will prove the result for collinear points, but the proof for concurrent lines only involves switching the roles of points and lines.  Let $P_1$ and $P_2$ be as in the statement of the lemma.  Without loss of generality, we may assume that $B_1 = (\delta, \delta, \ldots, \delta)$.  The graph $\Gamma$ is locally 2-arc-transitive, so $G_{P_1}$ is transitive on points collinear with $P_1$, and hence $G_{P_1}$ acts transitively on the blocks in which collinear points are contained.  If $B_2$ is Hamming distance $d$ from $B_1$, precisely $k-d$ of its coordinates are $\delta,$ and, since $G_{P_1}$, which fixes $B_1$, maps $B_2$ to any other block containing a point collinear with $P_1$, any block $B_i$ containing a point collinear with $P_1$ has exactly $k-d$ entries that are $\delta$.  Finally, since $N$ is transitive on $\mathcal{P}$, the result holds for any choice of $P_1$. 
\end{proof}

\begin{lem}
 \label{lem:prodstructonP}
Assume Hypothesis \ref{hyp:GQPA}.  Then, we may identify $\calP$ with the set $\Delta^k$, $|\calP| = |\Delta|^k = |T:T_\delta|^k$, and $N_P \cong T_\delta^k$.    
\end{lem}

\begin{proof}
 Assume that $\calQ$ is a thick locally $(G,2)$-transitive generalized quadrangle such that $G$ is quasiprimitive on both points and lines with type $\{\PA, \PA \}$.  We again use the notation of Hypothesis \ref{hyp:GQPA}.
 
 By \cite[Theorem 1.1]{primGQ}, $G$ must preserve a system of imprimitivity on either $\calP$ or $\calL$.  Suppose first that $s = t$.  If $G$ is primitive on either $\calP$ or $\calL$, then up to duality we may assume that $G$ is primitive on $\calP$.  This would mean that $|\calP| = |\Delta|^k$, and by \cite[Lemma 3.1]{primGQ} this implies that $|\Delta| = 20$, $k = 2$, and $s = 7$, which is ruled out by inspection.  Otherwise, $G$ is imprimitive on both $\calP$ and $\calL$, which by Lemma \ref{lem:blocks=t} implies up to duality that any blocks of points have size $s^2 + 1 = st + 1$.  If, on the other hand, $s < t$, then by Lemma \ref{lem:Brestrictions}, we may assume that all blocks of points have size $st+1$.  Thus, in any case, if $G$ preserves a system of imprimitivity on $\calP$, we may assume that the blocks all have size $st+1$.
 
 Let $B = (\delta, \delta, \dots, \delta)$.  If $B \neq \{P\}$, then this block must have size $st + 1$, which implies that $|\Delta|^k = s + 1$.  By Lemma \ref{lem:loc2trans}, $\Gamma$ is $G$-locally primitive, so the $s+1$ points incident with a line $\ell$ are in different blocks, i.e., each line is incident with a point in each block.  On the other hand, since $k \ge 2$, not all blocks are the same Hamming distance apart, which is a contradiction to Lemma \ref{lem:Hamming}.  Hence, $B = \{P\}$, and we may identify the points of $\calP$ with elements of $\Delta^k$, implying that $|\calP| = |\Delta|^k = |T:T_\delta|^k$, as desired.
\end{proof}

\begin{lem}
\label{lem:katmost4}
Assume Hypothesis \ref{hyp:GQPA}.  Then $k \le 4$. 
\end{lem}

\begin{proof}
 By Lemma \ref{lem:prodstructonP}, $G$ preserves a product structure on $\calP$.  The result now follows immediately from Lemma \ref{lem:PAdimrest}. 
\end{proof}

\begin{thm}
\label{thm:notPA}
 Let $\mathcal{Q}$ be a thick locally $(G,2)$-transitive generalized quadrangle such that $G$ is quasiprimitive on both points and lines.  
Then $G$ cannot act with type $\{\PA, \PA \}$.  
\end{thm}

\begin{proof}
 Assume Hypothesis \ref{hyp:GQPA}.  By Lemma \ref{lem:prodstructonP}, $|\calP| = |\Delta|^k = |T:T_\delta|^k$ and $N_P \cong T_\delta^k$, and by Lemma \ref{lem:katmost4}, $k \le 4$.
 
 Suppose first that $k = 2$.  If $G \le (H_1 \times H_2) \rtimes \Sym(2)$, then define $Y:= G \cap (H_1 \times H_2)$.  Thus $G/Y \cong G_P/Y_P \cong C_2$, which implies that $Y_P^{\Gamma(P)}$ is primitive.  On the other hand, $N_P \cong T_\delta^2$, and each of $T_\delta \times 1$ and $1 \times T_\delta$ are normal subgroups of $Y_P$.  Since nontrivial normal subgroups of primitive groups are transitive, this implies that all lines incident with $P$ are in a block with the same first coordinate and the same second coordinate, i.e., all lines incident with $P$ are in the same block, a contradiction to local primitivity.  Hence $k \neq 2$.
 
 Suppose now that $k = 3$.  Since $\Gamma$ is $G$-locally primitive, the $t+1$ lines incident with a point $P$ are in different blocks.  By Lemma \ref{lem:Hamming}, the $t+1$ blocks containing the lines incident with $P$ are all at Hamming distance $1$, $2$, or $3$ apart.  Assume first that they are all at Hamming distance $3$.  This implies that $|T:T_\epsilon| = |\Sigma| \ge t+1$, which means that, if $B'$ is the block containing the line $\ell$, then $$|\mathcal{L}| = |T:T_\epsilon|^3|N_{B'}:N_\ell| \ge (t+1)^3,$$ a contradiction to Lemma \ref{lem:Pbounds}(ii).   

Suppose next that two lines incident with $P$ are in blocks Hamming distance $1$ apart.  Without a loss of generality, let these two lines be  $\ell$ in $(\epsilon, \epsilon, \epsilon)$ and $\ell_2 \in B_2 = (\epsilon, \epsilon, \epsilon_1)$, where $\epsilon_1 \neq \epsilon \in \Sigma.$  There are two possibilities: either every line incident with $P$ is in a block of the form $(\epsilon, \epsilon, \epsilon')$, $\epsilon' \neq \epsilon$; or some line incident with $P$ is in a block one of whose first two coordinates is not $\epsilon$.  Consider the former case first.  Since no two lines are in the same block, and $P$ is incident with $t+1$ lines, we have $|T:T_\epsilon| = |\Sigma| \ge t+1$, a contradiction as above.  Now consider the second case.  Let $\ell_3$ be incident with $P$, and without a loss of generality we may assume that $\ell_3 \in B_3 = (\epsilon', \epsilon, \epsilon)$, where $\epsilon' \neq \epsilon$.  However, the Hamming distance between $B_2$ and $B_3$ is $1$, a contradiction.

Now, suppose that any two lines incident with $P$ are in blocks Hamming distance $2$ apart.  Without a loss of generality, let two of these lines be $\ell$ in $(\epsilon, \epsilon, \epsilon)$ and $\ell_2 \in B_2 = (\epsilon, \epsilon_2, \epsilon_3)$, where the $\epsilon_i \neq \epsilon \in \Sigma.$  As above, there are two possibilities: either every line incident with $P$ is in a block of the form $(\epsilon, \epsilon', \epsilon'')$, where $\epsilon', \epsilon'' \neq \epsilon$; or some line incident with $P$ is in block whose first coordinate is not $\epsilon$.  Consider the former case first, and let $\ell_3 \in B_3 = (\epsilon, \epsilon', \epsilon'')$, where $\epsilon', \epsilon'' \neq \epsilon$.  Since the Hamming distance between $B_2$ and $B_3$ is $1$, and they are both in blocks with an $\epsilon$ in the first coordinate, $\epsilon' \neq \epsilon_2$.  Hence no two lines incident with $P$ are in blocks whose second coordinate is the same, which as above implies that $|T:T_\epsilon| = |\Sigma| \ge t+1$, a contradiction for the same reasons as above.  Finally, we consider the second case, and without a loss of generality, we have $P$ incident with $\ell_3 \in B_3 = (\epsilon', \epsilon, \epsilon'')$, where $\epsilon', \epsilon'' \neq \epsilon.$  Since $B_2$ and $B_3$ must be Hamming distance $2$ apart, and they differ in the first two coordinates, this forces $\epsilon_3 = \epsilon''$.  However, this implies that there can be no other line incident with $P$ in a block whose first coordinate is $\epsilon$: the third coordinate would have to differ from $\epsilon_3$ to be concurrent with $\ell_2$, but it must be $\epsilon_3$ to be concurrent with $\ell_3$.  Similarly, there is no other line whose second coordinate is $\epsilon$: its third coordinate would have to differ from $\epsilon_3$ to be concurrent with $\ell_2$, but it must be $\epsilon_3$ to be concurrent with $\ell_3$.  The only remaining possibility is that there are lines in blocks whose last coordinate is $\epsilon$.  Suppose $\ell_4 \in (\epsilon_4, \epsilon_5, \epsilon)$, where $\epsilon_4, \epsilon_5 \neq \epsilon$.  For this line to be concurrent with $\ell_2$ and $\ell_3$, we must have $\epsilon_4 = \epsilon'$ and $\epsilon_5 = \epsilon_2$.  Hence there can be no other lines in a block whose last coordinate is $\epsilon$, and thus $t+1 \le 4$, a contradiction, since no generalized quadrangle with $t \le 3$ satisfies Hypothesis \ref{hyp:GQPA}.  Hence $k \neq 3$.

We are left now only with the case $k = 4$.  We proceed largely as in the $k = 3$ case, although the analysis is a bit more tedious.  Since $\Gamma$ is $G$-locally primitive, the $t+1$ lines incident with a point $P$ are in different blocks.  By Lemma \ref{lem:Hamming}, the $t+1$ blocks containing the lines incident with $P$ are all at Hamming distance $1$, $2$, $3$, or $4$ apart.  Assume first that they are all at Hamming distance $4$.  This implies that $|T:T_\epsilon| = |\Sigma| \ge t+1$, which means that, if $B'$ is the block containing the line $\ell$, then $$|\mathcal{L}| = |T:T_\epsilon|^4|N_{B'}:N_\ell| \ge (t+1)^4,$$ a contradiction to Lemma \ref{lem:Pbounds}(ii).

Suppose next that two lines incident with $P$ are at Hamming distance $1$ apart.  Without loss of generality, let these two lines be  $\ell$ in $(\epsilon, \epsilon, \epsilon, \epsilon)$ and $\ell_2 \in B_2 = (\epsilon, \epsilon, \epsilon, \epsilon_1)$, where $\epsilon_1 \neq \epsilon \in \Sigma.$  There are two possibilities: either every line incident with $P$ is in a block of the form $(\epsilon, \epsilon, \epsilon, \epsilon')$, $\epsilon' \neq \epsilon$; or some line incident with $P$ is in a block one of whose first three coordinates is not $\epsilon$.  Consider the former case first.  Since no two lines are in the same block, and $P$ is incident with $t+1$ lines, we have $|T:T_\epsilon| = |\Sigma| \ge t+1$, a contradiction as in the Hamming distance $4$ case.  Now consider the second case.  Let $\ell_3$ be incident with $P$, and without a loss of generality we may assume that $\ell_3 \in B_3 = (\epsilon', \epsilon, \epsilon, \epsilon)$, where $\epsilon' \neq \epsilon$.  However, the Hamming distance between $B_2$ and $B_3$ is $2$, a contradiction.

Now suppose that two lines incident with $P$ are at Hamming distance $2$ apart.  Without loss of generality, let these two lines be $\ell$ in $(\epsilon, \epsilon, \epsilon, \epsilon)$ and $\ell_2$ in $(\epsilon, \epsilon, \epsilon_1, \epsilon_2)$, $\epsilon_1, \epsilon_2 \neq \epsilon$.  Now, if either all lines incident with $P$ are in blocks whose first coordinate is $\epsilon$ or all lines incident with $P$ are in blocks whose second coordinate is $\epsilon$, then this case reduces to the $k = 3$, Hamming distance $1$ case, a contradiction.  Hence, for any given coordinate, we may find two lines incident with $P$ in blocks that differ in that coordinate.  Without loss of generality, let a third line $\ell_3$ incident with $P$ be in $(\epsilon_3, \epsilon, \epsilon, \epsilon_4)$, $\epsilon_3, \epsilon_4 \neq \epsilon$.  Since $\ell_3$ and $\ell_2$ are both incident with $P$, their respective blocks are at Hamming distance $2$ as well, and so $\epsilon_4 = \epsilon_2$, and $\ell_3$ is in the block $(\epsilon_3, \epsilon, \epsilon, \epsilon_2)$.  Since not all lines incident with $P$ are in a block with second coordinate $\epsilon$, there must exist a line $\ell_4$ in a block either of the form $(\epsilon, \epsilon_4, \epsilon_5, \epsilon)$ or $(\epsilon, \epsilon_4, \epsilon, \epsilon_5)$, $\epsilon_4, \epsilon_5 \neq \epsilon$.  Now, $\ell_4$ cannot be in a block of the form $(\epsilon, \epsilon_4, \epsilon_5, \epsilon)$, since this would imply $\ell_3$ and $\ell_4$ are in blocks at Hamming distance $4$.  Thus $\ell_4$ must be in a block of the form $(\epsilon, \epsilon_4, \epsilon, \epsilon_5)$.  Since $\ell_4$ is in a block at Hamming distance $2$ from the block containing $\ell_2$, we have $\epsilon_5 = \epsilon_2$, and so $\ell_4$ is in the block $(\epsilon, \epsilon_4, \epsilon, \epsilon_2)$.  Finally, we consider a fifth line $\ell_5$ incident with $P$.  It cannot be in a block of the form $(\epsilon_5, \epsilon_6, \epsilon, \epsilon)$, $(\epsilon, \epsilon_5, \epsilon_6, \epsilon)$, or $(\epsilon_5, \epsilon, \epsilon_6, \epsilon)$, $\epsilon_5, \epsilon_6 \neq \epsilon$, since then $\ell_5$ is at Hamming distance $4$ from $\ell_2$, $\ell_3$, or $\ell_4$, respectively.  If $\ell_5$ is in a block of the form $(\epsilon, \epsilon, \epsilon_5, \epsilon_6)$, where $\epsilon_5, \epsilon_6 \neq \epsilon$, then we have $\epsilon_5 \neq \epsilon_1$ and $\epsilon_6 \neq \epsilon_2$ so that $\ell_5$ and $\ell_2$ are at Hamming distance $2$; however, this means $\ell_5$ and $\ell_3$ are at Hamming distance $3$, a contradiction.  If $\ell_5$ is in a block of the form $(\epsilon, \epsilon_5, \epsilon, \epsilon_6)$, where $\epsilon_5, \epsilon_6 \neq \epsilon$, then we have $\epsilon_5 \neq \epsilon_4$ and $\epsilon_6 \neq \epsilon_2$ so that $\ell_5$ and $\ell_4$ are at Hamming distance $2$; however, this means $\ell_5$ and $\ell_2$ are at Hamming distance $3$, a contradiction.  Finally, if $\ell_5$ is in a block of the form $(\epsilon_5, \epsilon, \epsilon, \epsilon_6)$, where $\epsilon_5, \epsilon_6 \neq \epsilon$, then we have $\epsilon_5 \neq \epsilon_3$ and $\epsilon_6 \neq \epsilon_2$ so that $\ell_5$ and $\ell_3$ are at Hamming distance $2$; however, this means $\ell_5$ and $\ell_2$ are at Hamming distance $3$, a contradiction.  Hence, $t+1 \le 4$, a contradiction since no generalized quadrangle with $t \le 3$ satisfies Hypothesis \ref{hyp:GQPA}.

Finally, we suppose that two lines incident with $P$ are in blocks at Hamming distance $3$ apart.  Without loss of generality, let these two lines be $\ell$ in $(\epsilon, \epsilon, \epsilon, \epsilon)$ and $\ell_2$ in $(\epsilon, \epsilon_1, \epsilon_2, \epsilon_3)$, $\epsilon_1, \epsilon_2, \epsilon_3 \neq \epsilon$.  Again, if all lines incident with $P$ are in blocks whose first coordinate is $\epsilon$, then we have a contradiction as in the Hamming distance $4$ case.  Without loss of generality, we may assume that $\ell_3$ incident with $P$ is in block $(\epsilon_4, \epsilon, \epsilon_5, \epsilon_3)$, where $\epsilon_4, \epsilon_5 \neq \epsilon$ and $\epsilon_5 \neq \epsilon_2$.  Now, assume that there is another line $\ell^\prime$ incident with $P$ in a block whose first coordinate is $\epsilon$, say the block $(\epsilon, \epsilon_6, \epsilon_7, \epsilon_8)$, where $\epsilon_6, \epsilon_7, \epsilon_8 \neq \epsilon$ and $\epsilon_6 \neq \epsilon_1$, $\epsilon_7 \neq \epsilon_2$, and $\epsilon_8 \neq \epsilon_3$.  Since $\ell^\prime$ and $\ell_3$ are in blocks at Hamming distance $3$, this implies that $\ell_7 = \ell_5$.  Any other line $\ell^{\prime\prime}$ incident with $P$ in a block whose first coordinate is $\epsilon$ would also have third coordinate $\epsilon_5$, implying the Hamming distance between $\ell^\prime$ and $\ell^{\prime\prime}$ is less than $3$, a contradiction.  Hence, at most two lines incident with $P$ have first coordinate $\epsilon$.  Indeed, an analogous argument shows that at most two lines other than $\ell$ incident with $P$ are in a block with any single fixed coordinate equal to $\epsilon$, which shows that $t+1 \le 2\cdot4 + 1 = 9$.  Since $t \le 8$, we have
\[|\Sigma|^4 = |\calP| = (s+1)(st+1) \le (t+1)(t^2 + 1) \le 585, \]
and so $|\Sigma| < 5$.  However, the finite simple group $T$ must have a permutation representation on $|\Sigma|$ points, and no finite simple group has a permutation representation on fewer than five points, a contradiction.  Therefore, $G$ cannot act with type $\{\PA, \PA \}$.    
\end{proof}

We can now prove the following result, which characterizes locally $(G,2)$-transitive generalized quadrangles such that $G$ is quasiprimitive on both points and lines.

\begin{thm}
 \label{thm:quasi2}
 Let $\mathcal{Q}$ be a thick locally $(G,2)$-transitive generalized quadrangle such that $G$ is quasiprimitive on both points and lines.  Then, $G$ is an almost simple group.  
\end{thm}

\begin{proof}
 By Lemma \ref{lem:possactions}, if the incidence graph $\Gamma$ is locally $(G,2)$-arc-transitive and $G$ is quasiprimitive on both $\mathcal{P}$ and $\mathcal{L}$, if $G$ acts with quasiprimitive types $\{X,Y\}$, then either $X = Y \in \{\HA, \TW, \AS, \PA \}$ or $\{X,Y\} = \{\SD, \PA\}$.  The result follows from Theorem \ref{thm:QPonboth1} 
 and Theorem \ref{thm:notPA}.
\end{proof}

Finally, we can now prove Theorem \ref{thm:main_rephrased}. 

\begin{thm}
 \label{thm:main_rephrased}
 If $\mathcal{Q}$ is a thick locally $(G,2)$-transitive generalized quadrangle, then one of the following holds:
 \begin{enumerate}[(i)]
  \item $\mathcal{Q}$ has order $(3,5)$ or $(5,3)$. 
  \item $G$ is an almost simple group that is quasiprimitive on both points and lines.
 \end{enumerate}
\end{thm}

\begin{proof} 
 If $\mathcal{Q}$ is a thick locally $(G,2)$-transitive generalized quadrangle, then one of the following holds: $G$ is not quasiprimitive on either points nor lines; up to duality, $G$ is quasiprimitive on points but not on lines; or $G$ is quasiprimitive on both points and lines.  The result then follows by Theorems \ref{thm:quasi0poss}, \ref{thm:quasi1}, and \ref{thm:quasi2}.
\end{proof}

\section{Imprimitive on both points and lines}
\label{sect:imprim}

The purpose of this section is to prove that, if $\calQ$ is a locally $(G,2)$-transitive generalized quadrangle, then $G$ must be primitive on either points or lines.

\begin{prop}
 \label{prop:ASst+1}
Assume that $G$ is an almost simple group and $\calQ$ is a thick locally $(G,2)$-transitive generalized quadrangle.  Then $G$ cannot stabilize a partition of either points or lines with blocks of size $st+1$. 
\end{prop}

\begin{proof}
 Let $\calQ$ be a locally $(G,2)$-transitive generalized quadrangle, where $G$ is an almost simple group.  Without loss of generality, assume that $\mathcal{B}$ is a block system on  $\calL$, let $B \in \mathcal{B}$, and assume $|B| = st+1$.  This means $|G:G_B| = t + 1$, and $G$ has a $2$-transitive action on $\mathcal{B}$.  Since the size of each block is $st+1$, each block is a spread, and hence each point of $\calQ$ is incident with exactly one line in each block.  Moreover, if $P \in \calP$, $G_P$ has a $2$-transitive action on $t+1$ elements, and $G_P$ is $2$-transitive on $\mathcal{B}$.  By the Frattini Argument, $G = G_B G_P$ is a factorization of $G$, that is, every element of $G$ can be written as $xy$ for some $x \in G_B$ and $y \in G_P$.  We may assume that $G_P \le G_C < G$, where $C$ is a maximal block of points (possibly of size $1$) and $G_C$ is maximal in $G$.  (Note that, since $B$ is a spread, both $G_P$ and $G_C$ must have $2$-transitive actions on $\calB$.)  In particular, this implies that $G = G_B G_C$ is a maximal factorization, i.e., a factorization where each of the groups in the factorization is a maximal subgroup, and $G_B$ is the stabilizer of an element in a $2$-transitive action of $G$ on $t+1$ elements.  By the CFSG, the possibilities for such a $2$-transitive group $G$ with element stabilizer $G_B$ are known.  Furthermore, by \cite{maxfactor}, the possibilities for $G_C$ are also known.  Henceforth in this proof, let $T:= \soc(G)$ and let $G = T.A$, where $A \le \Out(T)$. 
 
 Consider first the case when $T = A_{t+1}$, the alternating group of degree $t+1$.  By \cite[Theorem 1.2(b)]{primGQ}, if $G$ is primitive on $\calP$, then $G \le S_6$ and $\calQ$ is the unique generalized quadrangle of order $(2,2)$.  However, there is no such system of imprimitivity in this case, so we may assume that $G$ is not primitive on points, i.e., we may assume that $G_P < G_C$.  We will assume henceforth that $t+1 > 6$, since the cases $t+1 = 5$ and $t+1 = 6$ are ruled out by inspection.  By Lemma \ref{lem:Brestrictions}, $|G_C:G_P| \ge s+1$, and, since $|G_P| \ge s(t+1)$, we have 
 \[ |G_C| = |G_P||G_C:G_P| \ge s(t+1)(s+1) > (s+1)(st+1) = |G:G_P| > |G:G_C|.\]
Hence $|G| < |G_C|^2$, and, since $t + 1 > 6$, we have $|T| < |T_C|^2 |A| \le 2|T_C|^2 < |T_C|^3,$ and so $T_C$ must be a large subgroup of $T$.  Moreover, $G_C$ is $2$-transitive on the $t+1$ cosets of $G_B$ in $G$, and hence $T_C$ is a large primitive maximal subgroup of $A_{t+1}$.  These are explicitly known by \cite[Theorem 2]{largesubs}.  Furthermore, we know additionally that $|T| < 2|T_C|^2$, $G_C$ has a $2$-transitive action on $t+1$ elements, and $G_P < G_C$ also has a $2$-transitive action on $t+1$ elements with $|G_C:G_P| \ge s+1$ and $s > 2$.  The remaining possibilities for $(t+1, T_C, T_P)$ are listed in Table \ref{st+1Alt}, and, in each case, we may use the equation

\[ (s+1)(st + 1) = |\calP| = |T:T_P|\]

to solve for $s$.

\begin{center}
\begin{table}[ht]
\caption{Remaining possibilities for $(t+1, T_C, T_P)$, in the case $T=A_{t+1}$.}\label{st+1Alt}
\begin{tabular}{l|l|l|l|l}
\toprule
$t+1$ & $T_C$ &$T_P$ &$|\calP|$ &$s$ \\
\midrule
$7$ &  $\PSL(2,7)$ & $\ASL(1,7)$ & $120$ & $\notin \N$  \\
$8$ & $\AGL(3,2)$ &   $\PSL(3,2)$ & $120$ & $\notin \N$  \\
$8$ &  $\AGL(3,2)$ &  $\AGammaL(1,8)$ & $120$ & $\notin \N$   \\
$8$ &  $\AGL(3,2)$ &  $\AGL(1,8)$ & $360$ & $\notin \N$   \\
$8$ &  $\AGL(3,2)$ &  $\AGL(1,8)$ & $360$ & $\notin \N$   \\
$11$&  $M_{11}$    &  $\PSL(2,11)$ & $30240$ & $\notin \N$  \\
$12$&  $M_{12}$    &  $M_{11}$    &  $30240$ & $\notin \N$  \\
$12$&  $M_{12}$    &  $\PSL(2,11)$    &  $362880$ & $\notin \N$  \\
\bottomrule
\end{tabular}
\end{table}
\end{center}
Hence, it is impossible for $T = A_{t+1}$.

Next, we consider the case when $T = \PSL(n,q)$.  In this case, $t + 1 = (q^n - 1)/(q-1)$ and $G_B \cong P_1$ (a maximal parabolic subgroup).  By \cite{maxfactor}, there are a few possibilities for $G_C$, where $G_P \le G_C$ and $G_C$ is maximal in $G$.  First, we could have $T_C \cong \frac{1}{\gcd(n,q-1)} \GL(a,q^b).b$, where $ab = n$ and $b$ is prime.  However, since $G_P \le G_C$ and $G_P$ is $2$-transitive on $t+1$ elements, $|G_C|$ must be divisible by a primitive prime divisor of $q^n - 1$ (unless $n = 6$ and $q = 2$).  Moreover, since $G_P$ has a $2$-transitive action on $t+1$ elements, $|G_C|$ is also divisible by $t = (q^n - q)/(q-1)$, and so $|G_C|$ must also be divisible by a primitive prime divisor of $q^{n-1} -1$, which it is not by \cite[Main Theorem]{GPPS}. Hence we must have $q = 2$ and $n = 6$.  This means $t + 1 = 63$, but $\PSL(6,2)$ does not contain a proper subgroup with a $2$-transitive action on $63$ points, ruling this case out.

We now suppose $T_C \cong \PSp(n,q)$, where $n$ is even and at least $4$.  We know that $|G_C|$ must be divisible by a primitive prime divisor of $q^n - 1$; on the other hand, since $G_P$ has a $2$-transitive action on $t+1$ elements, $|G_C|$ is also divisible by $t = (q^n - q)/(q - 1)$, and so $|G_C|$ must also be divisible by a primitive prime divisor of $q^{n-1} - 1$, which it is not by \cite[Main Theorem]{GPPS}, ruling out this subcase.

The only remaining cases have $n = 2$, and so $t = q$.  However, in each of these cases $q \le 7$, while $T_C \le S_5$, ruling these cases out.  Hence it is impossible for $T = \PSL(n,q)$.

By \cite{maxfactor}, there are no maximal factorizations of either $\Sz(q)$ or $\Ree(q)$, so $T$ cannot be either of these.

Next, we consider the case $T = \PSU(3,q)$.  In this case, $G_B = P_1$ and $t + 1 = q^3 + 1$.  Moreover, by \cite{maxfactor}, the only maximal factorizations occur when $q \in \{3,5,8\}$.  However, in none of these cases is there an integer solution to \[q^3 s + 1 = st + 1 = |G:G_C|,\] and so it is impossible for $T = \PSU(3,q)$.

The next case we consider is $G = T = \Sp(2n,2)$ for some $n \ge 2$, where $G_B = \POmega^-(2n, 2)$ and $t+1 = 2^{2n-1} - 2^{n-1}$.  Of the possibilities listed in \cite[Tables 1--3]{maxfactor}, we consider first the subcase when $G_C = \Sp(2a, 2^b).b$, where $ab = n$ and $b$ is prime.  Since $G_C$ does not have a $2$-transitive action on $t+1$ elements, this case is ruled out immediately.

Suppose next that $G_C = P_k$, a parabolic subgroup, where
\[ P_k = [2^{2nk - \frac{3k^2 - k}{2}}]\colon (\GL(k,2) \circ \Sp(2n-2k,2))\]
and $1 \le k < n$.  In this situation $G_C$ can only act $2$-transitively on $2^k$, $2^{2n-2k}$, $2^k - 1$, or $2^{2n - 2k - 1} \pm 2^{n - k - 1}$ points, none of which equals $t+1$, a contradiction.  For a similar reason, we may also rule out $G_C = \Sp(n,2) \Wr S_2$ when $n$ is even and $G_C = \POmega^+(2n,2)$. 

If $G_C = G_2(2) \cong \PSU(3,3).2$ and $n = 3$, then $t + 1 = 28$.  Since the only proper subgroup of $\PSU(3,3).2$ with a $2$-transitive action on $28$ points is $\PSU(3,3)$ and $s > 2$, we must have $G_P = G_C$.  This means 
\[(s+1)(28s +1) = |\calP| = |G:G_P| = 120,\]
a contradiction since $s$ is an integer, ruling this case out.

The final subcase is $n = 4$ and $G_C = S_{10}$.  In this case, $t+1 = 120$.  While $S_{10}$ has a primitive action on $120$ points, it is not $2$-transitive, ruling this case out.  Hence we cannot have $G = T = \Sp(2n,2)$ and $G_B = \POmega^-(2n, 2)$.

We next consider the case when $G = T = \Sp(2n,2)$ for some $n \ge 2$, $G_B = \POmega^+(2n, 2)$, and $t+1 = 2^{2n-1} + 2^{n-1} = 2^{n-1}(2^n + 1)$.  We proceed as in the previous case through the possibilities listed in \cite[Tables 1--3]{maxfactor}.  The cases when $G_C = \Sp(2a, 2^b).b$ or $G_C = \POmega^-(2n,2)$ are ruled out since these choices of $G_C$ do not have a $2$-transitive action on $t+1$ elements.  If $n = 3$ and $G_C = G_2(2) \cong \PSU(3,3).2$, then $G_C$ does not have a $2$-transitive action on $t+1 = 36$ elements.  Finally, if $n = 4$ and $G_C = \PSL(2,17)$, then $G_C$ does not have a $2$-transitive action on $t+1 = 136$ elements.  Hence $T \not\cong \Sp(2n,2)$.

We now consider the sporadic almost simple $2$-transitive actions.  As above, in each case $G$ must contain a maximal subgroup $G_C$ (which, in this case, cannot be conjugate in $G$ to $G_B$) such that $G_C$ has a $2$-transitive action on $t+1$ elements.  We summarize the possibilities in Table \ref{sporadic}.

\begin{center}
\begin{table}[H]
\caption{Remaining possibilities for $(T,G_C)$.}\label{sporadic}
\begin{tabular}{l|l|l}
\toprule
$T$ & $t+1$ & Possible $G_C$  \\
\midrule
$\PSL(2,11)$ & $11$ & $\varnothing$\\
$M_{11}$     & $11$ & $\PSL(2,11)$\\
$M_{11}$     & $12$ & $\varnothing$\\
$M_{12}$     & $12$ & $M_{11}$, $\PSL(2,11)$, $\PGL(2,11)$\\
$A_7$        & $15$ & $\varnothing$\\
$M_{22}$     & $22$ & $\varnothing$\\
$M_{23}$     & $23$ & $\varnothing$\\
$M_{24}$     & $24$ & $\PSL(2,23)$\\
$\PGammaL(2,8)$ & $28$ & $\varnothing$\\
${\rm HS}$   & $176$ & $\varnothing$ \\
${\rm Co}_3$ & $276$ & $\varnothing$\\
\bottomrule
\end{tabular}
\end{table}
\end{center}

We deal now with the remaining cases.  If $G = T = M_{11}$, then $t+1 = 11$, $G_B = M_{10}$, and $G_C = \PSL(2,11)$.  Since $\PSL(2,11)$ has no proper subgroups with a $2$-transitive action on $11$ elements, we conclude that $G_P = G_C$.  However, this means that
\[(s + 1)(10s + 1) = |\calP| = |G:G_P| = 12,\]
a contradiction.

Consider now the cases when $T = M_{12}$ and $t+1 = 12$.  If both $G_B$ and $G_C$ are isomorphic to $M_{11}$, then, since \[|G:G_P| = (s+1)(11s + 1) > 12 = (t+1) = |G:G_C|,\] we must have $G_P < G_C$.  The only proper subgroup of $M_{11}$ with a $2$-transitive action on $12$ elements is $\PSL(2,11)$, so we conclude that $G_P \cong \PSL(2,11)$, and so
\[ (s+ 1)(11s + 1) = |\calP| = |G:G_P| = 144,\] which is a contradiction to $s \in \N$.  If $G_C = \PSL(2,11)$, then, since $\PSL(2,11)$ has no proper subgroups with a $2$-transitive action on $12$ elements, we conclude that $G_P = G_C$ and reach a contradiction as in the previous case.  Finally, if $G_C = \PGL(2,11)$, then $G = M_{12}.2$ and, since $s > 2$, we again conclude that $G_P = G_C$ and reach a contradiction as in the previous cases.  Hence $T \not\cong M_{12}$.

Finally, we consider the case when $G = T = M_{24}$, $t+1 = 24$, and $G_C = \PSL(2,23)$.  Since $G_C$ has no proper subgroups with a $2$-transitive action on $24$ elements, we conclude that $G_P = G_C$.  However,
\[ (s+1)(23s + 1) = |\calP| = |G:G_P| = 40320\]
implies that $s \notin \N$, a final contradiction.  Therefore, if $G$ is an almost simple group and $\calQ$ is a locally $(G,2)$-transitive generalized quadrangle, $G$ cannot stabilize a block system with blocks of size $st + 1$ on either points or lines.
\end{proof}

%

\begin{thm}
 \label{thm:prim}
Let $\calQ$ be a thick locally $(G,2)$-transitive generalized quadrangle.  Up to duality, $G$ is primitive on points. 
\end{thm}

\begin{proof}
 Let $\calQ$ be a locally $(G,2)$-transitive generalized quadrangle.  By Theorem \ref{thm:main}, we may assume up to duality that $\calQ$ is not the unique generalized quadrangle of order $(3,5)$, in which case $G$ will be primitive on points, or $G$ is an almost simple group acting quasiprimitively on both points and lines.  By Lemmas \ref{lem:Brestrictions} and \ref{lem:blocks=t}, if $G$ is imprimitive on both points and lines, we know that $G$ must stabilize a block system with blocks of size $st + 1$ on either points or lines, which is impossible by Proposition \ref{prop:ASst+1}.  The result follows.
\end{proof}

\section{Reduction to large point stabilizers in groups of Lie type}
\label{sect:large}

Let $\calQ$ be a locally $(G,2)$-transitive generalized quadrangle.  By Theorem \ref{thm:prim}, up to duality, we may assume that $G$ is primitive on points.  Furthermore, by Theorem \ref{thm:main_rephrased}, if $\calQ$ is not the unique generalized quadrangle of order $(3,5)$ or its dual, we may assume that $G$ is an almost simple group.  Let $T := \soc(G)$.  The purpose of this section is to show that $T$ is a simple group of Lie type, and, if $P \in \calP$, then $T_P$ is a \textit{large} subgroup of $T$, by which we mean $|T| < |T_P|^3$.  

\begin{lem}[{\cite[Theorem 1.2]{primGQ}}]
\label{lem:ASnotAn}
If $G$ acts flag-transitively and point-primitively on $\calQ$ and $\soc(G) = A_n$ with $n \ge 5$, then $G \le S_6$, $\soc(G) \cong A_6 \cong \PSL(2,9)$, and $\calQ$ is the unique generalized quadrangle of order $(2,2)$. 
\end{lem}

\begin{lem}
\label{lem:ASnotSporadic}
If $\calQ$ is a locally $(G,2)$-transitive generalized quadrangle, then $\soc(G)$ cannot be a sporadic simple group. 
\end{lem}

\begin{proof}
 By Theorem \ref{thm:prim}, $G$ must act primitively on points.  By \cite[Table 8]{primGQ}, the possibilities for $G$, $s$, $t$, and $G_P$ are known, and for none of these choices does $G_P$ have a primitive action on $t+1$ elements.  The result follows. 
\end{proof}

\begin{prop}
 \label{prop:Lie}
If $\calQ$ is a thick locally $(G,2)$-transitive generalized quadrangle and $\calQ$ is not the unique generalized quadrangle of order $(3,5)$ or its dual, then $G$ is an almost simple group of Lie type.  Moreover, up to duality, we may assume that $G$ is primitive on points and $s \le t$. 
\end{prop}

\begin{proof}
 Assume that $\calQ$ is not the unique generalized quadrangle of order $(3,5)$ or its dual.  That $G$ is an almost simple group of Lie type follows from Theorem \ref{thm:main_rephrased} and Lemmas \ref{lem:ASnotAn} and \ref{lem:ASnotSporadic}.
 
 Assume now that $G$ is an almost simple group of Lie type that is primitive on points.  If $G$ is primitive on lines as well, then certainly we may assume $s \le t$.  Otherwise, suppose $G$ is primitive on points but not lines.  If $s > t$, then by Lemma \ref{lem:Brestrictions}, $G$ preserves a system of imprimitivity on $\calL$ with blocks of size $st + 1$, a contradiction to Proposition \ref{prop:ASst+1}.  Therefore, up to duality, $G$ is primitive on points and $s \le t$, as desired.
\end{proof}

\begin{lem}
\label{lem:largeGP}
Let $\calQ$ be a locally $(G,2)$-transitive generalized quadrangle, assume $2 < s \le t$, and let $P \in \calP$.  Then $|G| < |G_P|^\frac{5}{2}$.   
\end{lem}

\begin{proof}
First, since $\calQ$ is locally $(G,2)$-transitive, $G_P$ has a $2$-transitive action on $t+1$ elements, and hence $|G_P| \ge (t + 1)t$.  If $s < t$, then $s +1 \le t$ and so 
\[|G_P|^{1.5} \ge \left( t(t+1)\right)^{1.5} > t \cdot t(t+1) \ge (s+1) \cdot s(t+1) > (s+1)(st+1).\]
If $s = t$, then $s^3 - 2s^2 - 1 > 0$ when $s \ge 3$.  This implies $s^4 + s^3 > s^4 + 2s^2 + 1$, and so \[(s(s+1))^3 > ((s+1)(s^2 + 1))^2,\]
i.e., \[ |G_P|^{1.5} \ge \left(s(s+1)\right)^{1.5} > (s+1)(s^2 + 1).\]
In either case, we have 
\[|G| = |G_P|\cdot |G:G_P| = |G_P|\cdot |\calP| = |G_P|\cdot (s+1)(st+1) < |G_P|^{2.5},\]
as desired.
\end{proof}

Throughout the remainder of this section, $\mathcal{Q}$ is a thick locally $(G,2)$-transitive generalized quadrangle of order $(s,t)$. We also suppose
$G$ is almost simple of Lie type, and $\soc(G)=T$.

\begin{lem}\label{notlargebound}
Let $A$ be the outer automorphisms corresponding to the quotient of $G_PT$ by $T$,
and suppose $2<s\le t$. If $P$ is a point and $T_P$ is not a large subgroup of $T$, then
\[
|T_P|^3 \le |T| \le |T_P|^{5/2} |A|^{3/2}.
\]
and hence $|A|^{3}\ge|T_P|$. In particular, $|\Out(T)|^3\ge|T_P|$.
\end{lem}

\begin{proof}
First observe that
\begin{align*}
|T_P|&=|G_P\cap T|= \frac{ |G_P| |T| }{ | G_P T |}= \frac{ |G_P| }{ | G_P T \colon T|}= \frac{ |G_P| }{ | A |}.
\end{align*}
By Lemma \ref{lem:largeGP}, $|G|\le |G_P|^{5/2}$ and so $|T||A| \le (|T_P||A|)^{5/2}$. Therefore,
$|T|\le |T_P|^{5/2} |A|^{3/2}$. Finally, since $|T_P|^3\le |T|$, we have $|T_P|^{1/2}\le |A|^{3/2}$.
\end{proof}

\begin{lem}\label{lognbound}
Let $n=(s+1)(st+1)$ and suppose $2<s\le t$. Then, $|\Out(T)| \le 3\log n$.  Furthermore, let $P$ be a point, and suppose $T_P$ is not a large subgroup of $T$.  
\begin{enumerate}[(i)]
\item If $|A|\le \log n$, then $|T_P|\le 29409$ and $|T| \le 1.185181\times 10^{11}$.
\item If $|A|\le 2\log n$, then $|T_P|\le 484596$ and $|T|\le 9.08533\times 10^{14}$. Moreover, if $\log n < |A| \le 2\log n$, then $T$ is isomorphic to one of the following: $\PSU(d, 2^f)$, $\PSU(d, 3^f)$, or $\POmega^+(8,q)$, where $q \not\equiv 0 \pmod 3$.	
\item If $|A|\le 3\log n$, then $|T_P|\le 2289183$ and $|T|\le 1.13799\times10^{17}$.  Moreover, if $2\log n < |A| \le 3\log n$, then $T$ is isomorphic to either $\PSL(d,q)$ with $d>2$ or $\POmega^+(8,3^f)$.
\end{enumerate}
\end{lem}

\begin{proof}
Since $n=(s+1)(st+1) < (t+1)^3$ and $|T_P| \ge t+1$, we have $|T_P|>n^{1/3}$.
Therefore,
\[
|T| = n|T_P| > n^{4/3}
\]
and hence $|T_P|^3>n$. Suppose $|A|\le \alpha\log n$ where $\alpha\ge 1$.
By Lemma \ref{notlargebound}, we have
\begin{align*}
|T_P|^3&\le |T|\le  |T_P|^{5/2} |A|^{3/2} \\
&\le  |T_P|^{5/2} (\alpha \log n)^{3/2}\\
&< |T_P|^{5/2} (3\alpha \log |T_P|)^{3/2}
\end{align*}
and hence 
\begin{equation}
|T_P| <  27\alpha^3 (\log|T_P|)^3.
\end{equation}

For a given value of $\alpha$, this allows us to bound both $|T_P|$ and hence also $|T|$, and we have:

\begin{center}\footnotesize
\begin{tabular}{lll}
\toprule
$\alpha$ & $|T_P|$ & $|T|$\\ 
\midrule
1&29410&118518040738\\
2&484596&908532744261494\\
3&2289183&113798703080610442\\
\bottomrule
\end{tabular}
\end{center}


Finally, the assertions about the structure of $T$ in (ii) and (iii) follow directly from \cite[Lemma 7.7]{Guralnick:2017aa} and its proof.
\end{proof}

\begin{thm}\label{largeortable}
Let $\mathcal{Q}$ be a locally $(G,2)$-transitive generalized quadrangle of order $(s,t)$, where $2< s\le t$. Suppose
$G$ is almost simple of Lie type, and $\soc(G)=T$.
Then one of the following occurs:
\begin{enumerate}[(i)]
\item $T_P$ is large;
\item $T=\PSL(4,9)$ or $T=\PSU(5,4)$;
\item $T$ appears in the table below:
\end{enumerate}
\begin{center}\footnotesize
\begin{tabular}{ll}
\toprule
$\PSL(2,q)$ & $q \in\{ 27,64,125,169,243,289,343,512,529,729,1024,1369,1849,2187,2209,2809,3125,4489,5329\}$\\ 
$\,^2B_2(q)$ & $q\in\{8,32,128\}$ \\ 
$\PSL(3,q)$ & $q\in\{7,8,13,16,25,31,32,37,43,49,61,64,67,79,97,103,109,127,128\}$\\ 
$\PSU(3,q)$ & $q\in\{4,8,11,17,23,32\}$\\%
\bottomrule
\end{tabular}
\end{center}
\end{thm}

\begin{proof}
Let $n=(s+1)(st+1)$ and assume that $T_P$ is not large in $T$.
\begin{description}
\item[Case 1] Assume that $|\Out(T)|\le \log n$]. 
By Lemma \ref{lognbound}, $|T|\le 1.19\times 10^{11}$, and
by aid of computer, we can work out easily the possibilities for $T$:
\begin{center}
\begin{tabular}{ll|ll}
\toprule
$\PSL(6,2)$ & -- & $\PSU(6,2)$ &--\\ 
$\PSL(5,2)$ & -- & $\PSU(5,2)$ &--\\ 
$\PSL(4,q)$ & $q\le 4$ & $\PSU(4,q)$ & $q\le 4$\\ 
$\PSL(3,q)$ & $q\le 25$ & $\PSU(3,q)$ & $q\le 23$\\
$\PSL(2,q)$ & $5< q\le 6211$ &$G_2(q)$ & $q\le 5$ \\ 
$\PSp(8,2)$ & -- &$\,^2B_2(q)$ & $q\le 2^7$ \\
$\PSp(6,q)$ & $q\le 3$ & $\,^2G_2(q)$ & $q\le 3^3$ \\ 
$\PSp(4,q)$ & $2<q\le 13$ &$\,^2F_4(2)'$&--\\ 
$\POmega(7,3)$ & -- &$\,^3D_4(2)$&--\\ 
$\POmega^+(8,2)$ & -- &$\POmega^-(8,2)$ & -- \\
\bottomrule
\end{tabular}
\end{center}
By Lemma \ref{notlargebound}, the outer automorphism groups of the above groups would need to have
size at least $\mu^{1/3}$, where $\mu$ is the size of a smallest maximal subgroup of $T$. Thus the list reduces to
(at most) the following:

\begin{center}
\begin{tabular}{ll}
\toprule
$\PSL(2,q)$ & $5<q\le 6211$, $q$ not prime\\
$\,^2B_2(q)$ & $q\le 2^7$ \\
$\PSL(3,q)$ & $2<q\le 25$\\ 
$\PSU(3,q)$ & $2<q\le 23$\\
\bottomrule
\end{tabular}
\end{center}
Via examples, we elaborate on how we deduced this smaller table of simple groups.
For $\,^2G_2(q)$, the outer automorphism group has size $f$ where $q=3^f$. The smallest maximal subgroup is
$C_{q-\sqrt{3q}+1}:C_6$ (see \cite[p. 398]{BrayHoltRoney-Dougal}). So we require that $\log_3(q)^3 > 6(q-\sqrt{3q}+1)$, which is never true.
For $\POmega^+(8,2)$, we use GAP \cite{GAP4} to compute the maximal subgroups. The smallest one has size 14400; the outer automorphism
group has size 6. Likewise for $\POmega^-(8,2)$, the smallest maximal subgroup has size 168; the outer automorphism
group has size 2. The group $\POmega(7,3)$, has smallest maximal subgroup of size 13824; the outer automorphism
group has size 2.

In the case that $T=\PSL(2,q)$, $5<q\le 6211$, $q$ not prime, we can
refine the list of examples by considering when $|\Out(\PSL(2,q))|^3$ is larger than the size $\mu$ of the smallest non-large maximal subgroup.
For many values of $q$, we do not have non-large maximal subgroups. This refinement appears in Table \ref{tbl:logn_PSL2}.
\begin{table}[H]
\caption{Results for $\PSL(2,q)$.}\label{tbl:logn_PSL2}
\begin{center}\footnotesize
\begin{tabular}{p{3.5cm}ll}
\toprule
Values $q$ for which $|\Out(\PSL(2,q))|^3>\mu$& $|\Out(\PSL(2,q))| $ &$\mu$\\
27&6&12\\
64&6&60\\
125&6&60\\
169&4&60\\
243&10&12\\
289&4&60\\
343&6&168\\
512&9&504\\
529&4&60\\
729&12&360\\
1024&10&60\\
1369&4&60\\
1849&4&60\\
2187&14&12\\
2209&4&60\\
2809&4&60\\
3125&10&60\\
4489&4&60\\
5329&4&60\\
\bottomrule
\end{tabular}
\end{center}
\end{table}

Similarly, in the case that $T=\PSU(3,q)$, $2<q\le 25$, we can
refine the list of examples by considering when $|\Out(\PSL(3,q))|^3$ is larger than the size $\mu$ of the smallest non-large maximal subgroup.
See Table \ref{tbl:logn_PSU3}.

\begin{table}[H]
\caption{Results for $\PSU(3,q)$.}\label{tbl:logn_PSU3}
\begin{center}\footnotesize
\begin{tabular}{p{5cm}ll}
\toprule
Values $1< q\le 25$ for which $|\Out(\PSU(3,q))|^3>\mu$& $|\Out(\PSU(3,q))| $ &$\mu$\\
4&4&39\\
8&18&57\\
11&6&72\\
17&6&168\\
23&6&72\\
\bottomrule
\end{tabular}
\end{center}
\end{table}

For the remainder of the proof, suppose $|\Out(T)|> \log n$.

\item[Case 2]  Assume that $\log n < |\Out(T)| \le 2\log n$.
By Lemma \ref{lognbound}(ii), $|T|\le  9.08533\times 10^{14}$ and $T\cong \PSU(d,2^f)$, $\PSU(d,3^f), \POmega^+(8,q)$; the latter examples having $q\not\equiv 0\pmod{3}$.
This gives us the following simple groups:
\begin{table}[H]
\caption{Candidates for the case $\log n < |\Out(T)| \le 2\log n$.}\label{tbl:2logncase}
\begin{center}\footnotesize
\begin{tabular}{ll}
\toprule
$\PSU(3,q)$ & $q\in\{3,4,8,9,16,27,32,64\}$\\
$\PSU(4,q)$ & $q\in\{2,3,4,8,9\}$\\
$\PSU(5,q)$ & $q\in\{2,3,4\}$\\ 
$\PSU(6,2)$ & --\\  
$\PSU(7,2)$ & --\\  
$\POmega^+(8,2)$ & --\\
\bottomrule
\end{tabular}
\end{center}
\end{table}

We can further remove examples from the table by (i) removing those for which all of their
maximal subgroups are large, and (ii) noting that $|\Out(T)|^3>|T_P|$ (by Lemma \ref{notlargebound}).
For $\POmega^+(8,2)$, the smallest maximal subgroup has size 14400; the outer automorphism
group has size 6 (i.e., $S_3$), which is much less than the cube-root of 14400. 
For the unitary groups, we use \textsc{Magma} \cite{Magma} to work out which of those examples in Table \ref{tbl:2logncase} have $|\Out(T)|^3$ larger than
the size $\mu$ of the smallest maximal subgroup. In particular, $\PSU(d,q)$ does not arise as a candidate for $T$ if $d\ge 4$,
except possibly $\PSU(5,4)$, which we cannot handle directly by computer. In this instance, we resort
to \cite[Table 8.20]{BrayHoltRoney-Dougal}, which shows that the smallest maximal subgroup of $\PSU(5,4)$ has size $1025/5=205$.
The outer automorphism group of $\PSU(5,4)$ has size $20$, and so this case remains.
(Moveover, for $(d,q)\in\{(4,2),(4,3),(4,4),(5,2),(6,2)\}$, every maximal subgroup of $\PSU(d,q)$ is large.)
Table \ref{tbl:PSU3} gives a summary of what is left over in the Lie rank $1$ case.

\begin{table}[H]
\caption{Candidates $T=\PSU(3,q)$ satisfying $\log n < |\Out(T)| \le 2\log n$ and $|\Out(T)|^3>\mu$.}\label{tbl:PSU3}
\begin{center}\footnotesize
\begin{tabular}{p{3.5cm}ll}
\toprule
Values $q$ for which $|\Out(\PSU(3,q))|^3>\mu$& $|\Out(\PSU(3,q))| $ &$\mu$\\
4& 4&39\\
8& 18&57\\
32& 30&72\\
\bottomrule
\end{tabular}
\end{center}
\end{table}

\item[Case 3]  Assume that $2\log n < |\Out(T)| \le 3\log n$.
By Lemma \ref{lognbound}(iii), we have the following candidates for $T$: 
$\PSL(d,q)$ ($d>2$), $\POmega^+(8,3^f)$.
By Lemma \ref{lognbound}, $|T|\le 1.138\times10^{17}$, which by aid of computer, produces the following list of
remaining cases:

\begin{center}\footnotesize
\begin{tabular}{ll}
\toprule
$\POmega^+(8,3)$ & --\\
$\PSL(3,q)$ & $2\le q\le 121$\\
$\PSL(4,q)$ & $2\le q\le 13$\\
$\PSL(5,q)$ & $2\le q\le 5$\\
$\PSL(6,2)$ & $2\le q\le 3$\\
$\PSL(7,2)$ & --\\
\bottomrule
\end{tabular}
\end{center}

We can further remove examples from the table by recalling that $|\Out(T)|^3>|T_P|$. For instance $\PSL(5,2)$ has an outer automorphism
group of order $2$, yet the smallest maximal subgroup of $\PSL(5,2)$ has size $155$. 
For $\PSL(3,q)$, the only examples we have for which $|\Out(\PSL(3,q))|^3$ is larger than the size $\mu$ of the smallest maximal subgroup
are in Table \ref{tbl:3logn_PSL3}. Note that if $q$ is odd, then we must have $q\equiv 1\pmod{6}$, since then $|\Out(\PSL(3,q))|=6f$
(where $q=p^f$ for some prime $p$). Otherwise, $|\Out(\PSL(3,q))|=2f$, which is, for the values of $q$
we are considering, always smaller than $\mu^{1/3}$ (see \cite[Table 8.3]{BrayHoltRoney-Dougal}). 
For $q$ even, $q\ge 8$, we have $|\Out(\PSL(3,2^f)| \ge 2f\ge 6$ and 
$\mu =168$ (because $\PSL(3,2)$ is the smallest maximal subgroup in this case). The only candidate we have left for $\PSL(d,q)$, where $d>3$, is $\PSL(4,9)$ (n.b., $|\Out(\PSL(4,9))|=16$ and $\mu=3072$).
\begin{table}[H]
\caption{Candidates $T=\PSL(3,q)$ satisfying $2\log n < |\Out(T)| \le 3\log n$ and $|\Out(T)|^3>\mu$.}\label{tbl:3logn_PSL3}
\begin{center}\footnotesize
\begin{tabular}{p{3.5cm}ll}
\toprule
Values $q$ for which $|\Out(\PSL(3,q))|^3>\mu$& $|\Out(\PSL(3,q))| $ &$\mu$\\
\bottomrule
7&6&57\\
8&6&168\\
13&6&72\\
16&24&273\\
25&12&651\\
31&6&72\\
32&10&168\\
37&6&168\\
43&6&72\\
49&12&360\\
61&6&72\\
64&36&4161\\
67&6&72\\
79&6&72\\
97&6&72\\
103&6&72\\
109&6&168\\
127&6&168\\
128&14&168\\
\bottomrule
\end{tabular}
\end{center}
\end{table}
\end{description}
\end{proof}

\begin{thm}\label{thm:whenlargeTP}
 Let $\calQ$ be a locally $(G,2)$-transitive generalized quadrangle, where $2 < s \le t$ and $G$ is an almost simple group of Lie type with $\soc(G) = T$.  
 If $P$ is a point of $\calQ$, then $T_P$ is large, that is, $|T| < |T_P|^3$. 
\end{thm}

\begin{proof}
For each example where $T_P$ is not large in Theorem \ref{largeortable}, we work out when there exists a maximal subgroup $H$ of $T$
such that $|T:H|$ is the number of points of a $GQ(s,t)$ with $2<s\le t$, satisfying the divisibility condition and Higman inequality (Lemma \ref{lem:GQbasics}(ii),(iii)).
We are left with only three possibilities:

\begin{center}
\begin{tabular}{lll}
\toprule
$T$ & $s$ & $t$ \\
\midrule
$\PSL(2,64)$ & 11 & 33\\ 
$\PSU(3,8)$ & 27 & 45\\
$\PSU(3,17)$ & 203 & 205\\
\bottomrule
\end{tabular}
\end{center}

(Note: In order to find the possible degrees of primitive permutation representations for $Sz(128)$, we used 
\cite[Table 8.16]{BrayHoltRoney-Dougal}. Otherwise, Magma/GAP was used).
Now we simply have three permutation groups to check. First, $\PSL(2,64)$ in its action on 4368 elements has $A_5$ as its
point stabilizer. However, $A_5$ does not have a transitive action of degree $t+1=34$; so this case does not arise.
Secondly, $\PSU(3,8)$ has the point stabilizer $H$ of order $162$. In particular, $H$ 
does not have a transitive action on $t+1=46$ elements. So this case
does not arise. Finally, $\PSU(3,17)$ acting on $204^3$ elements has a point stabilizer $H$ of order 8489664. In particular, $H$
does not have a transitive action on $t+1=206$ elements, since $8489664 \pmod{206} = 198$.
\end{proof}

\section{Primitive on points only}
\label{sect:prim1}

The purpose of this section is to prove that the only locally $(G,2)$-transitive generalized quadrangle such that $G$ is primitive on points but not lines is the unique generalized quadrangle of order $(3,5)$.

\begin{prop}
\label{prop:PrimImprimAS}
 Suppose that $\calQ$ is a thick locally $(G,2)$-transitive generalized quadrangle, and, up to duality, suppose that $G$ is primitive on points but not on lines.  If $\calQ$ is not the unique generalized quadrangle of order $(3,5)$, then $s \le t$, $s \mid t$, and each block of lines has size $t + 1$.  Moreover, $G$ is an almost simple group of Lie type, and, if $T: = \soc(G)$ and $P$ is a point, then $|T| < |T_P|^3$. 
\end{prop}

\begin{proof}
 Since we are assuming $\calQ$ is not the unique generalized quadrangle of order $(3,5)$, we know that $G$ is an almost simple group.  This now follows immediately from Lemma \ref{lem:Brestrictions}, Proposition \ref{prop:Lie}, and Theorem \ref{thm:whenlargeTP}. 
\end{proof}

\begin{thm}
\label{thm:PrimImprim35}
 Let $\calQ$ be a thick locally $(G,2)$-transitive generalized quadrangle, and, up to duality, suppose that $G$ is primitive on points but not on lines.  Then $\calQ$ is the unique generalized quadrangle of order $(3,5)$.
\end{thm}

\begin{proof}
 Assume that $G$ and $\calQ$ are as in the statement.  By Proposition \ref{prop:PrimImprimAS}, we have $s \le t$, $s \mid t$, each block of lines has size $t+1$, and $G$ is an almost simple group of Lie type.  Moreover, if $T: = \soc(G)$, then $|T| < |T_P|^3$.  Henceforth in this proof $\calB$ will denote the system of imprimitivity of size $st+1$ on lines, $B$ will denote a block of lines, and $G_B$ is a maximal subgroup of $G$ satisfying 
 \[|G:G_B| = |T:T_B| = st + 1.\]
 
 We begin by noting that $G$ nearly satisfies all the hypotheses in \cite[Hypothesis 5.1]{BLS}.  As long as the proofs of \cite[Propositions 5.3, 5.4, 5.5, 5.6, 6.1, 7.5, 7.6, 8.1]{BLS} do not rely on $G$ acting primitively on $\calL$, they will hold in this case as well.  In particular, we only need new proofs in the subcases of \cite[Propositions 5.3, 5.4, 5.5, 5.6, 6.1, 7.5, 7.6, 8.1]{BLS} that specifically invoke $G_\ell$ being maximal in $G$.  For example, if $T$ is isomorphic to $\PSL(n,q)$ for some $n,q \ge 2$, then we proceed as in \cite[Proposition 5.3]{BLS}.  We begin with \cite[Case 1 of Proposition 5.3]{BLS}, which is located at the bottom of page 1562.  Our argument proceeds exactly the same as in \cite{BLS} until we reach a point in which the maximality of $G_\ell$ in $G$ was used in the proof.  Here, we first invoked the maximality of $G_\ell$ on page 1564, and we have the following assumptions:  $T = \PSL(n,q)$, 
 \[T_P \cong P_2=q^{2(n-2)}\colon {1\over\gcd(n,q-1)}(\GL(2,q)\circ\GL(n-2,q)),\]
 $T_P^{\Gamma(P)} \cong \PGL(n-2, q)$, $t+1 = (q^{n-2} - 1)/(q - 1)$, 
 \[ |\calP| = \frac{(q^n - 1)(q^{n-1}-1)}{(q^2 - 1)(q - 1)},\]
 and the flag stabilizer in $T$ is
 \[T_{P, \ell} = [q^{2(n-2)+(n-3)}]\colon \left(\frac{1}{\gcd(n,q-1)}(\GL(2,q)\circ(\GL(1,q)\times\GL(n-3,q))\right).\]  We must now alter the proof under the assumption that $G_\ell$ is not maximal in $G$, i.e., we are now assuming that $|G:G_B| = |T:T_B| = st + 1$ and $|G_B:G_\ell| = |T_B:T_\ell| = t + 1$.  Since $T_{P,\ell} < T_B$ and $T_B$ is a maximal subgroup of $T$, $T_B$ must be a parabolic subgroup of type $P_1$, $P_2$, or $P_3$.  However, 
\[ |T:T_B| = st + 1 < (s+1)(st + 1) = |T:T_P|,\]
and so it must be that $T_B = P_1$, and so 
\[st + 1 = {n\brack 1} = \frac{q^n - 1}{q-1}.\] 
This means 
\[s+1 = \frac{|\calP|}{|\calB|} = \frac{q^{n-1} - 1}{q^2 - 1},\]
and so $s = (q^{n-1} - q^2)/(q^2 - 1).$  However, $q^2$ divides $s$, whereas the highest power of $q$ dividing $t$ is $q$, a contradiction to $s \mid t$, ruling this case out.

Proceeding through the rest of the proof of \cite[Proposition 5.3]{BLS} on pages 1565--1569 (Cases 2--6), at no other point was the maximality of $G_\ell$ in $G$ invoked, and so these proofs hold in this context as well.
 
We now proceed through the cases here in the order they are encountered in \cite{BLS}.  Note that the basic results contained in Lemmas \ref{lem:GQbasics}, \ref{lem:ratio}, and \ref{lem:Pbounds} are used repeatedly in these proofs.

We next consider the cases when $T = \PSU(n,q)$ and $T_P$ is a geometric subgroup.  We first remark that the statement of \cite[Proposition 5.4]{BLS} contains a misprint: in case (i), the subgroup labeled as $T_P$ should be $T_\ell$ and vice versa\footnote{We also take the opportunity to amend another error in \cite{BLS}. In the introduction, the definition of \emph{Moufang} should be: for each path $(v_0,v_1,v_2,v_3)$, the group $G_{v_0}^{[1]}\cap G_{v_1}^{[1]} \cap G_{v_2}^{[1]}$
acts transitively on $\Gamma(v_3)\setminus \{v_2\}$.}.  

We proceed as in the proof of \cite[Proposition 5.4]{BLS} on page 1570 until we reach the case when $T_P = P_1$, $n = 5$, $|\calP| = (q^2 + 1)(q^5 + 1)$, 
\[T_P=q.q^6\colon {1\over \gcd(5,q+1)}(\GL(1,q^2)\circ\GU(3,q)),\]
and $t+1 = q^3 + 1$.  This implies that $s = q^2$ and so $|T:T_B| = st + 1 = q^5 + 1$.  However, by \cite[Table 5.2A]{KleidmanLiebeck}, $\PSU(5,q)$ does not have a permutation representation of size $q^5 + 1$, a contradiction.

Next, we consider the case (also on page 1570) when $q = 3$, $T_P = P_k$, where $2 \le 2k \le n$ and $n - 2k = 3$, and $G_\ell^{\Gamma(\ell)} = 2^6:\PSU(3,3)$.  We also assume that $\soc(G_P^{\Gamma(P)}) \neq \PSU(3,q)$.  Then we have two possible cases:

\begin{itemize}
 \item[(i)] $k = 1$, $G_P^{\Gamma(P)}$ is solvable, $T = \PSU(5,3)$, and $T_P = 3.3^6:(\GL(1,3^2) \circ \GU(3,3))$,
 \item[(ii)] $k \ge 2$, and $G_P^{\Gamma(P)}\rhd \PSL(k,3^2)$.
 Then $t+1=|\Gamma(P)|={9^k-1\over 9-1}$, and $s+1=2^6$.
\end{itemize}

In case (i), we would need $T_B$ with $|T:T_B| < |T:T_P|$, which is impossible when $T_P = P_1$.  In case (ii),  since $t \le s^2$, we have $(9^k - 9)/(9-1) \le 63^2$, which implies that $k \le 4$.  Since $n - 2k = 3$ and $2 \le 2k \le n$, we have $n = 7, 9 , 11$.  On the other hand, there must exist a maximal subgroup $T_B$ with index $st + 1$, and so by \cite[Table 5.2A]{KleidmanLiebeck},
\[ 63\left( \frac{3^{n-3} - 9}{8}\right) + 1 \ge \frac{(3^n + 1)(3^{n-1}-1)}{8},\]
which does not hold for any $n \ge 3$, ruling this case out.

The arguments are then exactly the same as those in \cite{BLS} until we reach page 1573 and the case when $T_P = P_3$, $n = 7$, and
\[T_P = P_3 = q^9.q^6\colon \frac{1}{\gcd(7,q+1)}\GL(3,q^2).\]
This means that either $t + 1 = q^6$ or $t+1 = q^4 + q^2 + 1$, since $T_P$ is $2$-transitive on $\Gamma(P)$.  If $t + 1 = q^6$, then $s+1 = q^6 |T_\ell|/|T_P|$.  Since $s+1$ is a positive integer, this means $q^{15}$ divides $|T_\ell|$, and, since $|T_B:T_\ell| = t + 1$, this means $q^{21}$ divides $|T_B|$.  Looking at the possibilities for $T_B$ \cite[Tables 8.37, 8.38]{BrayHoltRoney-Dougal}, and noting that $|T_B| > |T_P|$, we have that $T_B = P_1$ or $T_B = P_2$.  Now $T_B \neq P_1$, since $T_\ell$ cannot have a composition factor isomorphic to $\PSU(5,q)$, and so $T_B = P_2$.  Since $\PSU(3,q)$ is a composition factor of $T_B$ and $\PSU(3,q)$ has no subgroup with index dividing $q^6$, this means that $\PSU(3,q)$ must be a composition factor of $T_\ell$ and so $s + 1$ must be $q^3 + 1$ ($q = 3$ is ruled out by inspection).  However, this means $s = q^3 \nmid q^6 - 1$, a contradiction.  Otherwise, $t + 1 = q^4 + q^2 + 1$, and so $s + 1 = (q^4 + q^2 + 1)|T_\ell|/|T_P|,$ which implies that $q^{21}$ divides $|T_\ell|$ and hence divides $|T_B|$, and, as above, we conclude that $T_B = P_1$ or $P_2$. The case of $T_B = P_1$ is ruled out precisely as above, and so $T_B = P_2$.  We again rule out $q = 3$ by inspection, and so $s+1 = q^3 + 1$ as above.  However, this is again a contradiction to $s \mid t$.

We now consider the case (also on page 1573) when $n = 2k$ and
\[T_P = q^{k^2}\colon {1\over\gcd(2k,q+1)}\GL(k,q^2),\] where $2 \le k \le 5$.  In this case, \[|\calP| = (q^{2k-1} + 1)(q^{2k-3} + 1) \cdots (q+1). \]
Either $G_P^{\Gamma(P)}$ is affine of degree $t+1 = q^{2k}$ or almost simple with socle $\PSL(k,q^2)$ and degree $t+1 = (q^{2k} - 1)/(q^2 - 1)$.  We have $q^{k^2 + k(k-1)/2}$ divides $(t+1)|T_\ell| = |T_B|$.  By \cite[Tables 8.10, 8.11, 8.26, 8.27, 8.46, 8.47, 8.62, 8.63]{BrayHoltRoney-Dougal}, noting that $|T_B| > |T_P|$, we have that $T_B = P_j$ for some $j < k$.  First, when $k = 2$, by \cite[Table 5.2A]{KleidmanLiebeck}, $T$ does not have a permutation representation on fewer than $(q+1)(q^3 + 1) = |\calP|$ elements, a contradiction.  When $k \ge 3$, since $|T:T_B| = st + 1$ divides $|\calP|$ and $(q^{2k} - 1)/(q^2 - 1)$ divides $|T:T_B|$ but not $|\calP|$, we reach a contradiction.

After this, we proceed exactly as in \cite{BLS} until we reach the case when $T = \PSU(3,q)$, $T_P = \PSL(2,q)$, and $q = 7, 9, 11$, found on page 1576.  The cases when $t+1 = q+ 1$ are ruled out in \cite{BLS}, so we need only check when $t+1 = 7, 6, 11$ in each of these respective cases.  These are ruled out as in \cite{BLS}, but, since it was not made explicit there, we show details here.  When $q = 7$, we have $(s+1)(6s + 1) = 33712,$ which has no integer roots; when $q = 9$, we have $(s+1)(5s + 1) = 118260,$ which has no integer roots; and when $q = 11$, we have $(s+1)(10s + 1) = 107448$, which has no integer solutions.

Our final case we need to alter the proof for $T = \PSU(n,q)$ and $T_P$ a geometric maximal subgroup is $T= \PSU(4,3)$ and $T_P = \PSL(2,9).2$, found at the bottom of page 1576 of \cite{BLS}.  In this case, we only need to explicitly check $t + 1 = 6$, which implies that $(s+1)(5s+1) = 4536$, which has no integer solutions.

We must next check the cases when $T = \PSp(2n, q)$, and we proceed through the proof of \cite[Proposition 5.5]{BLS}, which begins on page 1577.  First, we consider the case (which begins at the bottom of page 1577) when $T_P^{\Gamma(P)}$ is solvable,
\[T_P = P_n = q^\frac{n(n+1)}{2}:\GL(n,q),\] 
\[ |\calP| = \prod_{i=1}^n(q^i + 1),\]  
$t+1$ must be a power of $q$, and, if $q = p^f$, we must be in one of the following subcases for the same reasons as in \cite{BLS}, where it is noted that the equation $p^{fn(n+1)/6} < 2nfp^f$ must be satisfied: $n = 4$, $p = 2$, $f = 1$; $n = 3$, $p = 2$, $f \le 4$; $n = 3$, $p = 3$, $f \le 2$; $n = 3$, $p = 5$, $f = 1$; or $n = 2$ and $q = 5,7,9$ with $s < q + 1$.  While it would be possible to rule out some cases by looking for a solvable $2$-transitive action, we proceed uniformly.  In each case, we determine the possible values of $t$ by noting that $(t+1) < |\calP| < (t+1)^3$ and $t+1$ is a power of $q$.  Next, we attempt to determine any integer solutions to $(s+1)(st+ 1)  = |\calP|$ by checking whether $(t+1)^2 + 4t(|\calP| - 1)$ is a perfect square.  In no case is there an integer solution for $s$.
 
We next consider the case (found on page 1578 of \cite{BLS}) when $T_P^{\Gamma(P)}$ is solvable, $n = 2$, $T_P = P_1$, $|\calP| = (q^{4} - 1)/(q - 1)$, and $t+1$ is a power of $q$, and $q = 5, 7, 9$.  We check for integer solutions for $s$ as in the previous cases and find that there are none.

We now consider the case when $T_P = P_k$, $T_P^{\Gamma(P)} \rhd q^k:\GL(k,q)$, $t + 1 = q^k$, and $k \le n < 7$, which begins at the top of page 1579 of \cite{BLS} (this is ``case (c)'' within Case 1).  In this case, we have \[(q-1).(\PGL(k,q)\times \PSp(2n-2k,q)) \le T_{P, \ell} < T_B.\]   Unless $k = 2$, $q = 3$, and $T_B \cong 2.(\PSp(2,3) \times \PSp(2n-2, 3))$, we have that $T_B$ is isomorphic to $T_P$, a contradiction.  If $k = 2$ and $q = 3$, then we have $n = 3$, since $n < (3k + 7)/4$ (see \cite{BLS}), and thus $|T:T_B| > |T:T_P|$, a contradiction.

We now consider the case when $T_P = P_k$, $t + 1 = (q^k - 1)/(q - 1)$, $\soc(T_P^{\Gamma(P)}) \cong \PSL(k,q)$, and $2 \le k = n \le 4$.  (This is part of ``case (d)'' within Case 1 on page 1579 of \cite{BLS}.)  If $n = k = 2$, then $t = q$ and $s = q$.  However, by \cite[Table 5.2A]{KleidmanLiebeck}, $\PSp(4,q)$ does not have a subgroup of index $q^2 + 1$.  The case when $k = n = 3$ is ruled out as in \cite{BLS}, and, finally, if $k = n = 4$, we may rule out all but when $q = 2$ as in \cite{BLS}.  When $q = 2$, then $t + 1 = 15$ and $|\calP| = 2295.$  However, $(s+1)(14s + 1) = 2295$ has no integer solutions, a contradiction.

Next, we have the case when $T_P$ is a $\mathcal{C}_1$-subgroup isomorphic to $\gcd(2,q-1).(\PSp(k,q) \times \PSp(2n-k,q))$, where $k$ is even.  This is ``Case 2,'' located toward the bottom of page 1579 of \cite{BLS}.   Whenever $q \neq 2$, we may assume that $k = 2$.  We need to rule out the cases when $k = 2$ and $q = 3, 4,5,7,8,9, 11$, which are done by examining $(s+1)(st+1) = |\calP|$ as above.  The cases when $n = k = 2$ with $t = q$ and $q = 2$ with $k \ge 3$ are ruled out as in \cite{BLS}.  Finally, we consider the case when $q = 2$, $k = 2$, and \[T_B > T_{P,\ell} = \Sp(2,2) \times \POmega^\epsilon(2n-2, 2).\]  Examining the possibilities for $T_B$, this implies that $|T:T_B| \ge |T:T_P|$, a contradiction to $|T:T_B| < |T:T_P|$.

We now consider the case when $T_P$ is a $\mathcal{C}_2$ subgroup isomorphic to $(q-1)/2.\PGL(n,q).2$, where $q$ is odd.  (This is part of ``Case 3,'' located on page 1580 of \cite{BLS}.)  Most cases are done as in \cite{BLS}, and the sporadic cases with $n = 2$ are ruled out as above by examining $(s+1)(st+1) = |\calP|$.

The final subgroups of $\PSp(2n,q)$ that need to be considered for $T_P$ are $\SO^+(4,q)$, where $q = 4,8$ when $n = 2$; $\SO^-(4,q)$, where $q = 2$ when $n = 2$; or $\SO^+(6,q)$, where $q = 2$ and $n = 3$.  Each of these cases, which correspond to the cases found in Case 7 on page 1581 of \cite{BLS}, is ruled out by inspection, again examining $(s+1)(st + 1) = |\calP|$ for the appropriate values of $t$.

Now assume that $T = \POmega^\epsilon(n,q)$, where $n \ge 7$.  These cases correspond to those found in \cite[Proposition 5.6]{BLS}.  Our first case is $n = 7$, $q$ odd, \[T_P = q^3.q^3.\GL(3,q),\] \[|\calP| = \frac{(q^6 -1)(q^4 - 1)(q^2 - 1)}{(q^3 - 1)(q^2 - 1)(q - 1)} = (q^3 + 1)(q^2 + 1)(q + 1),\] and either $t+1 = q^3$ or $t+1 = q^2 + q + 1$.  (This is part of ``Case 2,'' found on page 1583 of \cite{BLS}.)  Since $s+1$ divides $|\calP|$, $|\calP|$ is coprime to $q$, and $|T:T_B| = |\calP|/(s+1)$, we have that $|T:T_B|$ is coprime to $q$, and so $|T|_q = q^9$ divides $|T_B|$, implying that $T_B$ is a maximal parabolic subgroup.  Moreover, since $|T:T_B| < |T:T_P|$, we have that $T_B = P_1$ or $T_B = P_2$.  If $T_B = P_1$, then
\[ st + 1 = |T:T_B| = \frac{q^6 -1}{q-1} = (q^3 + 1)(q^2 + q + 1),\] implying that $q^2 + q + 1$ divides $(q^2 + 1)(q + 1)$ (since $st + 1$ divides $|\calP|$), a contradiction since
\[(q^2 + 1)(q + 1) = q(q^2 + q + 1) + 1.\]  If $T_B = P_2$, then 
\[st + 1 = |T:T_B| = \frac{(q^6 - 1)(q^4 - 1)}{(q^2 - 1)(q - 1)} = (q^2 + 1)(q^3 + 1)(q^2 + q + 1), \] which leads to a contradiction as in the $T_B  = P_1$ case. 

Next, we consider the cases when $T = \POmega^+(2d,q)$, \[ T_P = q^\frac{d(d-1)}{2} \cdot \frac{1}{q-1}\GL(d,q),\] \[|\calP| = 2(q^{d-1} + 1)(q^{d-1} + 1)\cdots(q + 1), \] and either $t + 1 = q^d$ or $t+1 = (q^d - 1)/(q - 1)$ for $d = 4,5,6$, which are found toward the bottom of page 1584 of \cite{BLS}.  If $d = 4$, then, since both $t + 1$ and $|\calP|$ are coprime to $q$, this means that $|T_B|_q = |T|_q = q^{12}$ and $T_B$ is a parabolic subgroup.  On the other hand, since $|T:T_B|$ divides $|T:T_P|$, we have a contradiction for each choice of parabolic subgroup (since $s > 1$).  When $d = 5, 6$, in each case, since $|T:T_B|$ divides $|\calP|$, we get that $|T_B|_q = |T|_q$, and so $T_B$ is a maximal parabolic subgroup.  We reach a contradiction as in the last case, since for no maximal parabolic subgroup do we get $|T:T_B|$ is a proper divisor of $|\calP|$.  

The final case among orthogonal groups is $T = \POmega^+(8,q)$, $T_P = O^+(4,q^2)$, and $q = 2,3$, which corresponds to Case 4(a) on page 1585 of \cite{BLS}.  In each case, there are no integer solutions to $(s+1)(st+1) = |\calP|$ in the sporadic cases that need checked.  

We can now conclude that $T_P$ is not a geometric maximal subgroup of a classical group, and the proof of \cite[Proposition 6.1]{BLS} does not require $T_\ell$ to be maximal in $T$, and so $T_P$ cannot be a $\mathcal{C}_9$-subgroup of $T$, either.

Among the novelty maximal subgroups of classical groups, which are considered in \cite[Section 7]{BLS}, there is only one case that needs to be reconsidered: $T = \PSL(n,q)$, $T_P$ is a $P_{m,n-m}$ type subgroup, where
\[T_P \cong [q^{2n-3}]:[a_{1,1,n-1}^+/\gcd(q-1,n)].(\PSL_1(q)^2
  \times \PSL_{n-2}(q)).[b_{1,1,n-2}^+],\] \[|\calP| = \frac{(q^n - 1)(q^{n-1} - 1)}{(q-1)^2}, \] $t + 1 = (q^{n-2} - 1)/(q - 1)$, and $s < \sqrt{2}q^{n/2}$ (see \cite[p. 1588]{BLS}, \cite[Proposition 4.1.22]{KleidmanLiebeck} for further details).  This implies that 
  \[T_B > T_\ell > T_{(\Gamma(P))} \ge [q^{2n - 3}]:\PSL(n-3, q),\] and so $T_B$ is a $P_m$ or $P_{m,n-m}$ type subgroup of $T$ with $m \le 3$.  Any choice of $T_B$ with $|T:T_B|$ dividing $|\calP|$ forces $s > \sqrt{2}q^{n/2}$, a contradiction.
  
Finally, we consider the possibility that $T$ is an exceptional group of Lie type (see \cite[Section 8]{BLS}), and there are only two cases that need to be considered.  The first case is $T = F_4(q)$, $q$ is even, $t + 1 = q^6$, \[T_P = (q^6 \times q^{1+8}): \Sp(6,q).(q-1),\] and \[|\calP| = \frac{(q^{12}-1)(q^8 - 1)}{(q^4-1)(q - 1)}. \] (This is part of ``Case (D4),'' found at the bottom of page 1597 of \cite{BLS}.)  However, $|\calP|$ is the size of the minimal permutation representation of $T$ by \cite{Vasil'ev}, a contradiction to $|T:T_B| < |\calP|$.

Our last case is $T = G_2(q)$, 
\[T_P = [q^5]:\GL(2,q),\] and \[|\calP| = \frac{q^6 - 1}{q - 1},\] which is ``Case (D8),'' found at the bottom of page 1598.  As in the previous case, by \cite{Vasil'ev} $|\calP|$ is the size of the minimal permutation representation of $T$, a contradiction to $|T:T_B| < |T|$.

Therefore, if $\calQ$ is a locally $(G,2)$-transitive generalized quadrangle and $G$ is primitive on points but not on lines, then $\calQ$ is the unique generalized quadrangle of order $(3,5)$. 
\end{proof}

%
%

\section{Primitive on points and lines} 
\label{sect:prim2}

Finally, we must consider the case when $\calQ$ is a locally $(G,2)$-transitive generalized quadrangle, where $G$ is an almost simple group of Lie type acting primitively on both points and lines.  This final case will allow us to prove Theorem \ref{thm:main}.

\begin{thm}
\label{thm:PrimPrimClassical}
Let $\calQ$ be a thick locally $(G,2)$-transitive generalized quadrangle, where $G$ is an almost simple group of Lie type acting primitively on both points and lines of $\calQ$, and let $T := \soc(G)$.  Then $\calQ$ is a classical generalized quadrangle.   
\end{thm}

\begin{proof}
 Assume $G$ is as in the statement.  Then, by Theorem \ref{thm:whenlargeTP}, a point stabilizer
 $T_P$ is large and $G$ satisfies \cite[Hypothesis 5.1]{BLS}.  We take this opportunity to correct a small misprint in the argument of \cite[Proposition 5.6]{BLS}.  In the case when $T = \POmega^+(2d, q)$, $4 \le k = d \le 8,$ \[T_P = q^{\frac{d(d-1)}{2}} \cdot \frac{1}{q-1}\GL(d,q), \] and either $t+1 = q^d$ or $t+1 = (q^d - 1)/(q-1)$, it should be
 \[ |\calP| = 2(q^{d-1} + 1)(q^{d-2} + 1) \cdots (q + 1).\]
However, the arguments used in the proof of \cite[Proposition 5.6]{BLS} otherwise remain exactly the same.  Also, in the statement of \cite[Proposition 5.4]{BLS} there is a misprint: in (i), the group that is specified as $T_P$ should be $T_\ell$, and vice versa.

By \cite[Propositions 5.3, 5.4, 5.5, 5.6, 6.1, 7.5, 7.6, 8.1]{BLS}, $\calQ$ is a classical generalized quadrangle.
\end{proof}

We can now prove the main result.

\begin{proof}[Proof of Theorem \ref{thm:main}]
 Let $\calQ$ be a locally $(G,2)$-transitive generalized quadrangle.  By Theorem \ref{thm:prim}, up to duality, $G$ is primitive on points.  The result now follows by Theorems \ref{thm:PrimImprim35} and \ref{thm:PrimPrimClassical}.
\end{proof}

\subsection*{Acknowledgements}

This paper forms part of an Australian Research Council Discovery Project (DP120101336).
The second author acknowledges the support of NSFC grants 11771200 and 11231008.
The authors are grateful to Michael Giudici for pointing out the valuable work of \cite{Guralnick:2017aa}, and the authors wish to thank the referees for extremely thorough and helpful reports.


\begin{thebibliography}{10}

\bibitem{largesubs}
S.~H. Alavi and T.~C. Burness.
\newblock Large subgroups of simple groups.
\newblock {\em J. Algebra}, 421:187--233, 2015.

\bibitem{AschbacherFusion}
M.~Aschbacher.
\newblock {$S_3$}-free 2-fusion systems.
\newblock {\em Proc. Edinb. Math. Soc. (2)}, 56(1):27--48, 2013.

\bibitem{primGQ}
J.~Bamberg, M.~Giudici, J.~Morris, G.~Royle, and P.~Spiga.
\newblock Generalised quadrangles with a group of automorphisms acting
  primitively on points and lines.
\newblock {\em J. Combin. Theory Ser. A}, 119:1479--1499, 2012.

\bibitem{GQtranshyp}
J.~Bamberg, S.~Glasby, T.~Popiel, and C.~Praeger.
\newblock Generalized quadrangles and transitive pseudo-hyperovals.
\newblock {\em J. Combin. Des.}, 24(4):151--164, 2016.

\bibitem{BLS}
J.~Bamberg, C.~H. Li, and E.~Swartz.
\newblock A classification of finite antiflag-transitive generalized
  quadrangles.
\newblock {\em Trans. Amer. Math. Soc.}, 370(3):1551--1601, 2018.

\bibitem{BambergPopielPraeger}
J.~Bamberg, T.~Popiel, and C.~E. Praeger.
\newblock Point-primitive, line-transitive generalised quadrangles of holomorph
  type.
\newblock {\em J. Group Theory}, 20(2):269--287, 2017.

\bibitem{BPP2}
J.~Bamberg, T.~Popiel, and C.~E. Praeger.
\newblock Simple groups, product actions, and generalized quadrangles.
\newblock {\em Nagoya Math J.}, 234:87--126, 2019.

\bibitem{BichMazzSomma}
A.~Bichara, F.~Mazzocca, and C.~Somma.
\newblock On the classification of generalized quadrangles in a finite affine
  space {${\rm AG}(3,\,2^{h})$}.
\newblock {\em Boll. Un. Mat. Ital. B (5)}, 17(1):298--307, 1980.

\bibitem{Magma}
W.~Bosma, J.~Cannon, and C.~Playoust.
\newblock The {M}agma algebra system. {I}. {T}he user language.
\newblock {\em J. Symbolic Comput.}, 24(3-4):235--265, 1997.
\newblock Computational algebra and number theory (London, 1993).

\bibitem{BrayHoltRoney-Dougal}
J.~N. Bray, D.~F. Holt, and C.~M. Roney-Dougal.
\newblock {\em The maximal subgroups of the low-dimensional finite classical
  groups}, volume 407 of {\em London Mathematical Society Lecture Note Series}.
\newblock Cambridge University Press, Cambridge, 2013.
\newblock With a foreword by Martin Liebeck.

\bibitem{DRG}
A.~E. Brouwer, A.~M. Cohen, and A.~Neumaier.
\newblock {\em Distance-regular graphs}, volume~18 of {\em Ergebnisse der
  Mathematik und ihrer Grenzgebiete (3) [Results in Mathematics and Related
  Areas (3)]}.
\newblock Springer-Verlag, Berlin, 1989.

\bibitem{BrownCherowitzo}
J.~M.~N. Brown and W.~E. Cherowitzo.
\newblock The {L}unelli-{S}ce hyperoval in {$\rm PG(2,16)$}.
\newblock {\em J. Geom.}, 69(1-2):15--36, 2000.

\bibitem{BuekenhoutHvM1994}
F.~Buekenhout and H.~Van~Maldeghem.
\newblock Finite distance-transitive generalized polygons.
\newblock {\em Geom. Dedicata}, 52(1):41--51, 1994.

\bibitem{CameronPerm}
P.~Cameron.
\newblock {\em Permutation Groups}.
\newblock Cambridge University Press, 1999.

\bibitem{localcompbip}
W.~Fan, D.~Leemans, C.~H. Li, and J.~Pan.
\newblock Locally 2-arc-transitive complete bipartite graphs.
\newblock {\em J. Combin. Theory Ser. A}, 120(3):683--699, 2013.

\bibitem{FongSeitz}
P.~Fong and G.~M. Seitz.
\newblock Groups with a {$(B,\,N)$}-pair of rank {$2$}. {I}, {II}.
\newblock {\em Invent. Math.}, 21:1--57; ibid. 24 (1974), 191--239, 1973.

\bibitem{GAP4}
The GAP~Group.
\newblock {\em {GAP -- Groups, Algorithms, and Programming, Version 4.10.0}},
  2018.

\bibitem{localsarc}
M.~Giudici, C.~Li, and C.~Praeger.
\newblock Analysing finite locally $s$-arc transitive graphs.
\newblock {\em Trans. Amer. Math. Soc.}, 356(1):291--317, 2004.

\bibitem{localstar}
M.~Giudici, C.~Li, and C.~Praeger.
\newblock Characterizing finite locally $s$-arc transitive graphs with a star
  normal quotient.
\newblock {\em J. Group Theory}, 9:641--658, 2006.

\bibitem{localdifferent}
M.~Giudici, C.~Li, and C.~Praeger.
\newblock Locally $s$-arc transitive graphs with two different quasiprimitive
  actions.
\newblock {\em J. Algebra}, 299:863--890, 2006.

\bibitem{CFSG}
D.~Gorenstein, R.~Lyons, and R.~Solomon.
\newblock {\em The Classification of Finite Simple Groups}, volume~40 of {\em
  A.M.S. Math. Surveys and Monographs}.
\newblock AMS, 1994.

\bibitem{GPPS}
R.~Guralnick, T.~Penttila, C.~E. Praeger, and J.~Saxl.
\newblock Linear groups with orders having certain large prime divisors.
\newblock {\em Proc. London Math. Soc. (3)}, 78(1):167--214, 1999.

\bibitem{Guralnick:2017aa}
R.~M. Guralnick, A.~Mar\'{o}ti, and L.~Pyber.
\newblock Normalizers of primitive permutation groups.
\newblock {\em Adv. Math.}, 310:1017--1063, 2017.

\bibitem{Kantor1991}
W.~M. Kantor.
\newblock Automorphism groups of some generalized quadrangles.
\newblock In {\em Advances in finite geometries and designs ({C}helwood {G}ate,
  1990)}, Oxford Sci. Publ., pages 251--256. Oxford Univ. Press, New York,
  1991.

\bibitem{KleidmanLiebeck}
P.~Kleidman and M.~Liebeck.
\newblock {\em The subgroup structure of the finite classical groups}, volume
  129 of {\em London Mathematical Society Lecture Note Series}.
\newblock Cambridge University Press, Cambridge, 1990.

\bibitem{LiSeressSong}
C.~H. Li, {\'A}.~Seress, and S.~J. Song.
\newblock {$s$}-arc-transitive graphs and normal subgroups.
\newblock {\em J. Algebra}, 421:331--348, 2015.

\bibitem{maxfactor}
M.~Liebeck, C.~Praeger, and J.~Saxl.
\newblock The maximal factorizations of the finite simple groups and their
  automorphism groups.
\newblock {\em Mem. Amer. Math. Soc.}, 86(432):1--151, 1990.

\bibitem{MSV}
L.~Morgan, E.~Swartz, and G.~Verret.
\newblock On 2-arc-transitive graphs of order {$kp^n$}.
\newblock {\em J. Combin. Theory Ser. B}, 117:77--87, 2016.

\bibitem{Ostrom:1959ys}
T.~G. Ostrom and A.~Wagner.
\newblock On projective and affine planes with transitive collineation groups.
\newblock {\em Math. Z}, 71:186--199, 1959.

\bibitem{fgq}
S.~Payne and J.~Thas.
\newblock {\em Finite Generalized Quadrangles}.
\newblock European Mathematical Society, 2009.

\bibitem{FinQuasiprimGraphs}
C.~Praeger.
\newblock Finite quasiprimitive graphs.
\newblock In {\em Surveys in Combinatorics, 1997}, pages 65--85. Cambridge
  University Press, 1997.

\bibitem{quasiprim1}
C.~E. Praeger.
\newblock An {O}'{N}an-{S}cott theorem for finite quasiprimitive permutation
  groups and an application to 2-arc transitive graphs.
\newblock {\em J. London Math. Soc.}, 47:227--239, 1993.

%

\bibitem{Song}
S.~J. Song.
\newblock On the stabilisers of locally 2-transitive graphs.
\newblock \url{http://arxiv.org/abs/1603.08398}.

\bibitem{titsngon}
J.~Tits.
\newblock Sur la trialit\'{e} et certains groupes qui s'en d\'{e}duisent.
\newblock {\em Inst. Hautes Etudes Sci. Publ. Math.}, 2:14--60, 1959.

\bibitem{ToborgWaldecker}
I.~Toborg and R.~Waldecker.
\newblock Finite simple {$3^\prime$}-groups are cyclic or {S}uzuki groups.
\newblock {\em Arch. Math. (Basel)}, 102(4):301--312, 2014.

\bibitem{Vasil'ev}
A.~V. Vasil'ev.
\newblock Minimal permutation representations of finite simple exceptional
  groups of types {$G_2$} and {$F_4$}.
\newblock {\em Algebra i Logika}, 35(6):663--684, 752, 1996.

\bibitem{Walter}
J.~H. Walter.
\newblock The characterization of finite groups with abelian {S}ylow
  {$2$}-subgroups.
\newblock {\em Ann. of Math. (2)}, 89:405--514, 1969.

\bibitem{Wilson}
R.~A. Wilson.  
\newblock \textit{The finite simple groups.}
\newblock Graduate Texts in Mathematics, vol. 251, Springer-Verlag London, Ltd., London, 2009.

\end{thebibliography}
\end{document}